\documentclass[pdflatex,sn-mathphys-num]{sn-jnl}

\usepackage{graphicx}
\usepackage{amsmath,amssymb,amsfonts}
\usepackage{amsthm}
\usepackage{mathrsfs}
\usepackage{bm}
\usepackage{xcolor}
\hypersetup{
  pdftitle={High-order energy-stable BGN parametric finite element methods for geometric flows},
  pdfauthor={Shu Ma and Qiqi Rao},
  pdfkeywords={parametric finite element method, BGN method, Runge--Kutta method, geometric flow, energy stability}
}

\theoremstyle{thmstyleone}
\newtheorem{theorem}{Theorem}[section]
\newtheorem{lemma}{Lemma}[section]
\newtheorem{proposition}{Proposition}[section]

\theoremstyle{thmstyletwo}

\newtheorem{remark}{Remark}[section]
\newtheorem{example}{Example}[section]
\numberwithin{equation}{section}

\def\X{{\bf X}}
\def\w{{\bf w}}
\def\n{{\bf n}}
\def\rhobm{\bm{\rho}}
\def\v{{\bf v}}
\def\u{{\bf u}}
\def\x{{\bf x}}
\def\d{{\mathrm d}}
\def\R {{\mathbb R}}
\def\id{{\bf  {id}}}

\graphicspath{{images/}}
\begin{document}

\title[High-order energy-stable BGN methods]{High-order energy-stable BGN parametric finite element methods for geometric flows}

\author[1]{\fnm{Shu} \sur{Ma}}\email{shu\_ma@hkbu.edu.hk}
\author*[2]{\fnm{Qiqi} \sur{Rao}}\email{qiqirao@cuhk.edu.hk}

\affil[1]{\orgdiv{Department of Mathematics}, \orgname{Hong Kong Baptist University},
\orgaddress{\city{Kowloon Tong}, \state{Hong Kong}, \country{China}}}
\affil*[2]{\orgdiv{Department of Mathematics}, \orgname{The Chinese University of Hong Kong},
\orgaddress{\city{Shatin}, \state{Hong Kong}, \country{China}}}

\abstract{
We construct high-order Runge--Kutta extensions of Barrett--Garcke--N\"urnberg (BGN) parametric finite element methods for curve-shortening flow and curve diffusion of planar curves and for mean curvature flow and surface diffusion of closed genus-$0$ surfaces. On each time slab $I_m=(t_m,t_{m+1}]$, the continuous equations are posed on the left-endpoint surface $\Gamma^m=\Gamma^{t_m}$ through the map $\X^{m,t}:\Gamma^m\to\Gamma^t$, whose target is the evolving surface at time $t$. For curves, this formulation follows by harmonic pullback from $\Gamma^0$ to $\Gamma^m$. For surfaces, an orientation-preserving harmonic diffeomorphism is posed separately on each time slab, and conformality yields the weight $\frac12|\nabla_{\Gamma^m}\X^{m,t}|^2$. Evaluating the time-slab equations at the Runge--Kutta internal times and applying mass-lumped parametric finite elements yields systems on the common domain $\Gamma^m$, with $\Gamma^{t_m+c_i\tau_m}$ as the intermediate target geometry. This internal-time discretization constructs high-order BGN-structured extensions for the four flows. For an algebraically stable tableau with nonnegative weights, every exact stage solution with nondegenerate intermediate configurations satisfies monotone decay of the discrete curve length or surface area for every positive time step for which that solution exists. Radau IIA experiments exhibit energy decay for all four flows and BGN-type mesh redistribution in the curve tests. The Hausdorff self-convergence results exhibit high-order behavior consistent with the corresponding design orders.
}



\maketitle

\section{Introduction}

Curvature-driven evolution equations describe interfaces whose normal velocity is determined by curvature and its surface derivatives. Let $\Gamma^t\subset\R^d$, $d=2,3$, be a smooth closed curve or surface at time $t$ with outward unit normal $\n$. Curve-shortening flow and mean curvature flow satisfy, respectively,
\begin{align}\label{eq:intro_mcf}
\v\cdot\n=-\kappa,\qquad \Delta_{\Gamma^t}\id=-\kappa\n
\quad\text{and}\quad
\v\cdot\n=-H,\qquad \Delta_{\Gamma^t}\id=-H\n,
\end{align}
whereas curve diffusion and surface diffusion prescribe
\begin{align}
\v\cdot\n=\Delta_{\Gamma^t}\kappa
\qquad\text{and}\qquad
\v\cdot\n=\Delta_{\Gamma^t}H.
\end{align}
The second-order flows decrease curve length or surface area, while the fourth-order flows decrease the same energy and preserve enclosed area or volume at the continuous level. The geometric laws leave the tangential velocity undetermined. Although tangential motion does not change the smooth evolving set, it determines node redistribution and mesh quality in a parametric discretization.

Parametric finite element methods (PFEMs) approximate the evolving interface directly by a polygonal curve or a triangulated surface. Dziuk's method for mean curvature flow is an early example \cite{dziuk1990algorithm}, and B\"ansch, Morin, and Nochetto developed a parametric finite element treatment of surface diffusion \cite{banschMorinNochetto2005parametric}. Barrett, Garcke, and N\"urnberg (BGN) introduced mixed variational formulations for a broad class of second- and fourth-order geometric flows \cite{barrett2007variational,barrett2007parametric,barrett2008parametric}, with further developments reviewed in \cite{bgn2020review}. For curve-shortening flow, let $\X(\cdot,t):\Gamma^0\to\Gamma^t$ be a flow map with $\X(\cdot,0)=\id_{\Gamma^0}$. The continuous system underlying the BGN formulation can be written as
\begin{subequations}
\begin{align}
\partial_t\X\cdot(\n\circ\X)&=-\kappa\circ\X
    &&\text{on }\Gamma^0,\\
-\Delta_{\Gamma^t}\id&=\kappa\n
    &&\text{on }\Gamma^t.\label{eq:intro_bgn_b}
\end{align}
\end{subequations}
The first equation fixes the normal motion. After spatial discretization, the weak curvature equation also selects a tangential nodal velocity. Their coupling produces the characteristic BGN mesh redistribution and, for polygonal curves under suitable assumptions, asymptotic equidistribution. The same mixed structure yields fully discrete length or area estimates without a restriction coupling the time step to the mesh size \cite{barrett2007parametric,barrett2008parametric}.

Convergence analysis must track both the evolving geometry and its parametrization. Results are available for semidiscrete and linearly implicit approximations of curve-shortening flow \cite{dziuk1994convergence,li2020dziuk,yeCui2021dziuk}, for evolving finite element approximations of mean curvature flow of closed surfaces \cite{kovacsLiLubich2019mcf,baiLi2024analysis}, and for a stabilized BGN-type curve-shortening method whose analysis identifies the limiting tangential trajectories \cite{baiLi2025bgn}.

Harmonic maps provide a systematic description of tangential redistribution. Harmonic-map heat flow and the DeTurck trick have been used to construct parametric algorithms with controlled mesh properties \cite{elliottFritz2016mesh,elliottFritz2017deturck}. Artificial tangential velocities based on harmonic or minimal-deformation principles have also been developed and analyzed in evolving finite element methods \cite{hu2022evolving,baiHuLi2024artificial}. Duan and Li introduced a fixed-reference formulation in which, for a fixed reference manifold $\Sigma$, the flow map $\X(\cdot,t):\Sigma\to\Gamma^t$ satisfies
\begin{subequations}\label{eq:intro_harmonic_map}
\begin{align}
\partial_t\X\cdot(\n\circ\X)&=(\v\cdot\n)\circ\X,\\
-\Delta_\Sigma\X&=(\lambda\n)\circ\X,
\end{align}
\end{subequations}
with scalar multiplier $\lambda$ \cite{duan2024new}. Subsequent harmonic-map and minimal-deformation formulations combine energy estimates with tangential velocities designed for surface-mesh redistribution \cite{duan2024energy,duan2025mesh,gaoLi2025geometric,gaoGarckeLiTang2026mdr}.

For curves, the fixed-reference system \eqref{eq:intro_harmonic_map} is pulled back from $\Gamma^0$ to each time-node curve $\Gamma^m$. For surfaces, the corresponding harmonic-map system is posed separately on each time slab. Write $\v^t=\v(\cdot,t)$, and let $\n^t$ be the unit normal to $\Gamma^t$. Let $0=t_0<t_1<\cdots<t_M=T$, set $\Gamma^m=\Gamma^{t_m}$, and write $\X^{m,t}:\Gamma^m\to\Gamma^t$ for $t\in(t_m,t_{m+1}]$. Thus $\Gamma^m$ is the left-endpoint surface and common domain of the time-slab equations, whereas $\Gamma^t$ is their time-dependent target. For closed genus-$0$ surfaces, conformality of the orientation-preserving harmonic diffeomorphism \cite{heleinWood2008harmonic} gives the multiplier equation
\begin{subequations}\label{eq:intro_time_slab_surface}
\begin{align}
\partial_t\X^{m,t}\cdot(\n^t\circ\X^{m,t})
&=(\v^t\cdot\n^t)\circ\X^{m,t}
&&\text{on }\Gamma^m,\label{eq:intro_time_slab_surface_a}\\
-\Delta_{\Gamma^m}\X^{m,t}
&=\frac12|\nabla_{\Gamma^m}\X^{m,t}|^2
  (H^t\n^t)\circ\X^{m,t}
&&\text{on }\Gamma^m,\label{eq:intro_time_slab_surface_b}\\
\X^{m,t_m}&=\id_{\Gamma^m}.
\end{align}
\end{subequations}
For planar curves, the corresponding curvature equation is
\begin{align}
-\Delta_{\Gamma^m}\X^{m,t}
=|\nabla_{\Gamma^m}\X^{m,t}|^2
  [ (\kappa^t\n^t)\circ\X^{m,t}].
\end{align}
For the fourth-order flows, \eqref{eqn_Dekat_Dem} and \eqref{eqn_Dekat_Dem_surf} pull $\Delta_{\Gamma^t}\kappa^t$ and $\Delta_{\Gamma^t}H^t$ back to $\Gamma^m$. Combining these identities with the four normal-velocity laws yields the paired weak systems \eqref{MCF-BGN-High-weak}, \eqref{sd-BGN-High-weak}, \eqref{MCF-BGN-High-weak_surf}, and \eqref{sd-BGN-High-weak_surf}. 

Freezing the metric factors and normals at the left endpoint and replacing the velocity by a backward difference gives a first-order explicit linearization of these time-slab equations. With mass-lumped piecewise linear finite elements, the resulting linear coupled systems are precisely the corresponding classical BGN schemes.

Classical fully discrete BGN schemes are predominantly first-order in time. Recent higher-order developments include a second-order BGN formulation for curve flows \cite{jiang2024second}, BGN schemes based on backward differentiation formulas (BDFs) of orders two through four \cite{jiangSuZhang2024bdf}, and a second-order predictor--corrector method for curve and surface diffusion that retains the long-time mesh behavior of the first-order method \cite{jiangSuZhangZhang2025predictor}. Second-order schemes based on strictly parabolic formulations with tangential motion have been analyzed for curve shortening and curve diffusion \cite{deckelnickNurnberg2026second}, while a recent harmonic-map-based construction treats structure-preserving surface evolution \cite{duanYang2026second}.

For curve and surface diffusion, fully discrete area or volume conservation has been enforced through time-averaged normals and related weak formulations \cite{jiangLi2021area,bao2021structure}, and through Lagrange multipliers combined with Crank--Nicolson or BDF time stepping \cite{garckeJiangSuZhang2025lagrange}. Isoparametric elements provide higher-order spatial discretizations with energy and volume properties \cite{garckeNurnbergPraetoriusZhang2025isoparametric}. High-order energy-diminishing time discretizations based on averaged-vector-field collocation were studied in \cite{duanLiZhang2021highorder}, and an arbitrary-order space--time structure-preserving construction using auxiliary variables was proposed in \cite{zhangAndrewsFarrell2026arbitrary}.

We discretize the time-slab systems with high-order algebraically stable implicit Runge--Kutta methods. Their algebraic energy identity supplies the quadratic relation used for energy-decaying systems \cite{butcher1975stability,burrageButcher1979stability,hairer2006solving}. Related energy results for finite-dimensional gradient systems are discussed in \cite{hairerLubich2014energy}. Implicit Runge--Kutta discretizations of parabolic equations on prescribed evolving surfaces have been analyzed in \cite{dziukLubichMansour2012,kovacs2018higher}.
Here the evolving geometry is itself an unknown. Evaluating the time-slab equations at the internal times produces paired relations on $\Gamma^m$, and their mass-lumped parametric finite element discretization gives the BGN-structured stage systems. Testing each velocity--curvature pair with the stage curvature and velocity gives a stagewise energy-decay identity, and algebraic stability yields the one-step discrete energy inequality
\[
|\Gamma_h^{m+1}|\le |\Gamma_h^m|,
\]
where $|\Gamma_h^m|$ denotes the polygonal length for the curve flows and the polyhedral surface area for the surface flows. Thus the discrete energy is nonincreasing at every time step.

The contributions of the paper are as follows.
\begin{itemize}
	\item We derive a common-domain time-slab formulation of the four geometric flows on the fixed left-endpoint geometry \(\Gamma^m\). For curves, this is achieved by pulling back the harmonic-map system, while for closed genus-\(0\) surfaces it follows from an intervalwise conformal harmonic diffeomorphism. This formulation places all Runge--Kutta stages on the same finite element space and provides the continuous foundation for retaining the paired BGN structure at high order in time.

\item We construct high-order Runge--Kutta extensions of the BGN method for curve-shortening flow, curve diffusion, mean curvature flow, and surface diffusion. The construction evaluates the continuous velocity--curvature pairs at the Runge--Kutta internal times and discretizes them by mass-lumped parametric finite elements on the common mesh $\Gamma_h^m$, retaining the stagewise BGN cancellation that drives discrete energy decay.

\item We prove time-step-unrestricted discrete energy decay for all four flows: whenever the nonlinear stage system admits a nondegenerate solution, any algebraically stable Runge--Kutta method satisfies
\[
|\Gamma_h^{m+1}|\le |\Gamma_h^m|,
\]
where $|\Gamma_h^m|$ denotes the polygonal length or polyhedral surface area.

\item Numerical experiments with the Radau IIA family demonstrate monotone discrete energy decay for all four flows and the corresponding design orders in Hausdorff shape distance. The curve experiments also retain the characteristic BGN mesh redistribution.
\end{itemize}

The paper is organized as follows. Section~\ref{Se:2} introduces the parametric finite element setting, algebraically stable Runge--Kutta methods, and the algebraic Runge--Kutta energy identity. Section~\ref{Se:3} derives the time-slab harmonic-map formulation and the corresponding BGN extensions for curve-shortening flow and curve diffusion, together with their length-decay estimates. Section~\ref{Se:4} develops the conformal time-slab formulation and the area-decaying schemes for mean curvature flow and surface diffusion of closed genus-$0$ surfaces. Section~\ref{Se:5} presents the temporal-convergence, energy, and mesh-quality experiments, and Section~\ref{Se:6} concludes the paper.

\section{PFEM and Runge--Kutta methods}\label{Se:2}
\subsection{Parametric finite element method}
Let $\Gamma^0\subset\R^d$, $d=2,3$, be a closed smooth initial curve or surface. Choose nodes $p_j\in\Gamma^0$, $j=1,\ldots,N$, together with a quasi-uniform simplicial topology, and collect them in $\x^0=(p_1,\ldots,p_N)\in\R^{dN}$. These nodes and their connectivity determine the polygonal or polyhedral interpolant $\Gamma_h^0:=\Gamma_h[\x^0]$ with simplicial mesh $\mathcal T_h^0$ and mesh size $h$. The associated linear finite element space is
\begin{align}\label{FEMS_on_init}
S_h[\x^0]  := \Big\{ v \in C^0 (\Gamma_h [\x^0]): v|_{K}\; \; \text{is linear},\; \; \forall K \in \mathcal{T}_h^0 \Big\}.
\end{align}

The nodal vector is evolved with the velocity $\v$ of $\Gamma^t$. Denote its value at time $t$ by $\x(t)=(x_1(t),\cdots,x_N(t))$, where
\begin{align}
    \frac{\d}{\d t} x_j(t) = \v(x_j(t),t), \qquad x_j(0)=p_j, \qquad j=1,\ldots,N.
\end{align}
There is a unique finite element function $\X_h(\cdot,t)\in S_h[\x^0]^d$ interpolating these nodal values:
\begin{align}
\X_h(p_j,t)=x_j(t) \quad \mbox{for} \quad j=1,\dots,N.
\end{align}
This function is the discrete flow map, which maps $\Gamma_h[\x^0]$ to the evolving curve or surface
$\Gamma_h [\x (t)] := \X_h (\Gamma_h [\x^0], t)$. Because $\X_h(\cdot,t)$ is piecewise linear, $\Gamma_h[\x(t)]$ is the piecewise linear curve or surface determined by $\x(t)$. Analogously to \eqref{FEMS_on_init}, the finite element space transported to this mesh is
\begin{align}
 S_h[\x (t)]  := \Big\{ v \in C^0 (\Gamma_h [\x (t)]): v \circ \X_h (\cdot,t)|_{K}\; \; \text{is linear},\; \; \forall K \in \mathcal{T}_h^0 \Big\},
\end{align}

Let $\mathcal T_h^t$ denote the transported simplicial mesh of $\Gamma_h^t:=\Gamma_h[\x(t)]$. For piecewise continuous scalar- or vector-valued functions $u$ and $v$, define the mass-lumped inner product by the elementwise vertex rule
\begin{align}\label{mass-lumped-inner-product}
\big(u,v\big)^h_{\Gamma_h^t}
:=\sum_{K\in\mathcal T_h^t}\frac{|K|}{d}
\sum_{\ell=1}^{d}(u\cdot v)|_K(p_{K,\ell}),
\end{align}
where $p_{K,\ell}$, $\ell=1,\ldots,d$, are the vertices of $K$. For curves in $\R^2$, this is the composite trapezoidal rule on polygonal edges. For surfaces in $\R^3$, it is the three-vertex quadrature rule on each triangle.

\subsection{Runge--Kutta methods and algebraic stability}
We discretize time with an implicit Runge--Kutta (RK) method. For
\begin{align}
    \frac{\d}{\d t} y (t) = f (y,t), \quad t \in (0,T], \quad y (0) = y_0,
\end{align}
let $0=t_0<t_1<\cdots<t_M=T$ be a partition of $[0,T]$ with step sizes $\tau_m=t_{m+1}-t_m$. An $s$-stage RK method has Butcher tableau $(c,A,b)$, where $A=(a_{ij})_{i,j=1}^s\in\R^{s\times s}$ and $b=(b_1,\dots,b_s)^{\top}$, $c=(c_1,\dots,c_s)^{\top}\in\R^s$. Its stage values and update are
\begin{align}\label{RK-stage}
    y^{m,i} = y^m + \tau_m \sum_{j=1}^{s} a_{ij} f\big(y^{m,j}, t_m + c_j \tau_m\big), \qquad i = 1, \dots, s,
\end{align}
\begin{align}\label{RK-update}
    y^{m+1} = y^m + \tau_m \sum_{i=1}^{s} b_i f\big(y^{m,i}, t_m + c_i \tau_m\big).
\end{align}
The method is consistent when $\sum_i b_i=1$ and $\sum_j a_{ij}=c_i$ for every $i$.

We use algebraically stable RK methods in the sense of Butcher: $b_i\ge0$ for all $i$ and
\begin{align}\label{alg-stable}
    \mathcal M := DA + A^{\top}D - b b^{\top}
\end{align}
is positive semidefinite ($\mathcal M\succeq0$), where $D=\mathrm{diag}(b_1,\dots,b_s)$. The three classical families considered here are
\begin{itemize}
\item the \textbf{Gauss--Legendre} methods, for which $\mathcal M=0$,
\item the \textbf{Radau IIA} methods, which are $L$-stable and satisfy $\mathcal M\succeq0$,
\item the \textbf{Lobatto IIIC} methods, which are $L$-stable and satisfy $\mathcal M\succeq0$.
\end{itemize}
Throughout the paper, $k$ denotes the classical order of the RK method. For the $s$-stage Radau IIA methods used in Section~\ref{Se:5}, $k=2s-1$.

The following algebraic identity will be applied to the geometry-dependent Runge--Kutta stage systems developed in Sections~\ref{Se:3} and~\ref{Se:4}.

\begin{lemma}[Algebraic Runge--Kutta identity]\label{lem:B-stability}
Let $V$ be a real Hilbert space with inner product $\langle\cdot,\cdot\rangle$ and induced norm $\|\cdot\|$. With $f^{m,i}$ denoting the stage slopes in \eqref{RK-stage}--\eqref{RK-update}, the following identity holds in $V$:
\begin{align}\label{B-stability-identity}
    \|y^{m+1}\|^2 - \|y^m\|^2
    = 2\tau_m \sum_{i=1}^{s} b_i \langle y^{m,i}, f^{m,i}\rangle
    - \tau_m^2 \sum_{i,j=1}^{s} \mathcal M_{ij} \langle f^{m,i}, f^{m,j}\rangle,
\end{align}
where $\mathcal M$ is the matrix in \eqref{alg-stable}.
\end{lemma}

\begin{proof}
Expanding the square and using the stage relation to eliminate $\langle y^m,f^{m,i}\rangle$ gives
\begin{align*}
\|y^{m+1}\|^2 - \|y^m\|^2
&= 2\tau_m \sum_{i=1}^{s} b_i \langle y^m, f^{m,i}\rangle
    + \tau_m^2 \sum_{i,j=1}^{s} b_i b_j
      \langle f^{m,i}, f^{m,j}\rangle\\
&= 2\tau_m \sum_{i=1}^{s} b_i
    \langle y^{m,i}, f^{m,i}\rangle\\
&\quad - 2\tau_m^2 \sum_{i,j=1}^{s} b_i a_{ij}
    \langle f^{m,j}, f^{m,i}\rangle
    + \tau_m^2 \sum_{i,j=1}^{s} b_i b_j
    \langle f^{m,i}, f^{m,j}\rangle\\
&= 2\tau_m \sum_{i=1}^{s} b_i \langle y^{m,i}, f^{m,i}\rangle\\
&\quad - \tau_m^2 \sum_{i,j=1}^{s}
    \big( b_i a_{ij} + b_j a_{ji} - b_i b_j \big)
    \langle f^{m,i}, f^{m,j}\rangle,
\end{align*}
where symmetry of the inner product and the relabeling $(i,j)\mapsto(j,i)$ combine the double sums. This is \eqref{B-stability-identity}. \end{proof}

\begin{remark}
If the RK method is algebraically stable and every stage satisfies $\langle y^{m,i},f^{m,i}\rangle\le0$, Lemma~\ref{lem:B-stability} gives $\|y^{m+1}\|\le\|y^m\|$. The first sum in \eqref{B-stability-identity} is nonpositive, and the second sum is nonnegative because $\mathcal M\succeq0$ and $(\langle f^{m,i},f^{m,j}\rangle)$ is a Gram matrix. This is the norm-decay consequence of algebraic stability under the stated stage condition \cite{hairer2006solving}.
\end{remark}

\section{RK-based PFEMs for curve flows}\label{Se:3}
Let $\Gamma^t\subset\R^2$ be a planar curve. We transport the harmonic-map system \eqref{eq:intro_harmonic_map} from $\Gamma^0$ to an arbitrary time-node curve $\Gamma^m$. The resulting equations are evaluated at the Runge--Kutta internal times on the common domain $\Gamma^m$ and expose the metric-weighted BGN pair.

\subsection{Time-slab harmonic-map formulation for curves}
Suppose that $\Gamma^t$ is a smooth embedded closed curve evolving with velocity $\v^t:=\v(\cdot,t)$. Let $\rhobm^t:\Gamma^t\to\R^2$ be a consistently oriented unit tangent vector field, and define
\begin{align}
\partial_{\rhobm^t}:=\rhobm^t\cdot\nabla_{\Gamma^t}.
\end{align}
Throughout this subsection, assume that $\X^t:\Gamma^0\to\Gamma^t$ is continuously differentiable in time and twice continuously differentiable in space, is orientation preserving, satisfies $\X^0=\id_{\Gamma^0}$, and has $|\partial_{\rhobm^0}\X^t|>0$. Harmonic parametrizations of a closed curve are invariant under phase shifts. We select one representative by prescribing the image of a marked material point. This normalization identifies the smooth pullback map and does not enter the differential identities below. The harmonic-map formulation is
\begin{subequations}\label{pde_X_kapa}
\begin{align}\label{pde-X-t0}
\partial_t \X^t \cdot (\n^{t}\circ \X^t)  & =  (\v^t \cdot \n^{t})\circ \X^t\quad  \mbox{on} \quad \Gamma^0,\\
\label{pde-kappa-t0}
- \Delta_{\Gamma^0} \X^t & = (\lambda^{t} \n^{t}) \circ \X^t\quad  \mbox{on} \quad \Gamma^0,
\end{align}
\end{subequations}
where $\lambda^{t}:\Gamma^{t}\to\R$ is a scalar-valued Lagrange multiplier. At fixed $t$, the second equation is the Euler--Lagrange equation for the Dirichlet energy
\begin{align}
E[\X^t] = \frac{1}{2} \int_{\Gamma^0} |\nabla_{\Gamma^0} \X^t|^2 \, \d \Gamma^0,
\end{align}
subject to the pointwise target constraint $\X^t(p)\in\Gamma^t$. Equation~\eqref{pde-X-t0} prescribes the normal evolution, whereas the target constraint in the second equation determines the multiplier. For curve-shortening flow, $\lambda^t$ can be expressed in terms of the curvature $\kappa^t$.

Because $\Gamma^{t}$ is generated by $\X^t$, the vector field $\partial_{\rhobm^{0}}\X^t\circ (\X^t)^{-1}$ is tangential on $\Gamma^{t}$. Thus,
\begin{align}\label{tan_vec_s}
\rhobm^t = \frac{\partial_{\rhobm^{0}} \X^t }{| \partial_{\rhobm^0} \X^t |} \circ (\X^t)^{-1}.
\end{align}
For any scalar-valued smooth function $f$ on $\Gamma^{t}$, the chain rule and \eqref{tan_vec_s} give
\begin{align}\label{rel_pt_rho_s}
\partial_{\rhobm^0} (f \circ \X^t) = &\; (\rhobm^0)^{\top} \nabla_{\Gamma^0} (f \circ \X^t) \notag \\
\phantom{\partial_{\rhobm^0} (f \circ \X^t) = }= &\; (\rhobm^0)^{\top} \nabla_{\Gamma^0}(\X^{t})^{\top}(\nabla_{\Gamma^{t}}f)\circ \X^{t} \notag \\
\phantom{\partial_{\rhobm^0} (f \circ \X^t) = }= &\; \partial_{\rhobm^0}\X^{t} \cdot (\nabla_{\Gamma^{t}}f)\circ \X^{t} \notag \\
\phantom{\partial_{\rhobm^0} (f \circ \X^t) = }= &\; | \partial_{\rhobm^{0}}\X^{t} | \rhobm^{t} \circ \X^{t} \cdot (\nabla_{\Gamma^{t}}f)\circ \X^{t} \quad \quad  \text{by using \eqref{tan_vec_s}} \notag \\
\phantom{\partial_{\rhobm^0} (f \circ \X^t) = }= &\; | \partial_{\rhobm^{0}}\X^{t} | (\partial_{\rhobm^{t}}f )\circ \X^{t}.
\end{align}
Since $\nabla_{\Gamma^{t}}f$ is tangential,
\begin{align}\label{eqn_tant_gradf}
    \nabla_{\Gamma^{t}}f = (\rhobm^{t}\cdot \nabla_{\Gamma^{t}}f) \rhobm^{t} = (\partial_{\rhobm^{t}}f) \rhobm^{t}.
\end{align}
Applying the tangential derivative once more to \eqref{eqn_tant_gradf} yields
\begin{align}\label{Tan_lap_f}
    \Delta_{\Gamma^{t}}f = \partial_{\rhobm^{t}\rhobm^{t}} f.
\end{align}
Applying these identities componentwise to a vector-valued smooth function $\u$ on $\Gamma^{t}$ gives
\begin{subequations}
    \begin{align}
\nabla_{\Gamma^{t}} \u = &\; (\partial_{\rhobm^{t}}\u) (\rhobm^{t})^{\top}, \label{t_grad}\\
\Delta_{\Gamma^{t}} \u = &\; \partial_{\rhobm^{t}\rhobm^{t}}\u. \label{lap_bel_t}
    \end{align}
\end{subequations}

\begin{lemma}\label{lemma:isometric}
Under the diffeomorphism and nondegeneracy assumptions above, if $\X^t$ satisfies \eqref{pde-kappa-t0}, then it is a constant-speed parametrization, i.e.,
\begin{align}
\partial_{\rhobm^{0}}  \big|\partial_{\rhobm^{0}} \X^{t} \big|   = 0,
\end{align}
which implies that $ | \nabla_{\Gamma^0} \X^{t} | $ is constant for fixed $t$.
\end{lemma}
\begin{proof}
Direct differentiation gives
\begin{align}\label{pt_ro_abX}
\partial_{\rhobm^{0}} | \partial_{\rhobm^{0}} \X^{t} | = \frac{\partial_{\rhobm^{0}}\X^t}{| \partial_{\rhobm^{0}}\X^{t} |}\cdot \partial_{\rhobm^{0} \rhobm^{0}}\X^{t}.
\end{align}
For the curve $\Gamma^0$, applying \eqref{lap_bel_t} with $t=0$ componentwise to $\X^t$ gives
\begin{align}
     \partial_{\rhobm^{0} \rhobm^{0}} \X^{t} = \Delta_{\Gamma^0} \X^{t}.
\end{align}
Using \eqref{pde-kappa-t0},
\begin{align}\label{pt_roro_X}
\partial_{\rhobm^{0} \rhobm^{0}} \X^{t}  = - (\lambda^{t} \n^{t}) \circ \X^{t} .
\end{align}
Substitution of \eqref{pt_roro_X} and \eqref{tan_vec_s} into \eqref{pt_ro_abX} gives
\begin{align}\label{est_ptro0X}
\partial_{\rhobm^{0}} | \partial_{\rhobm^{0}} \X^{t} | = - (\lambda^{t} \rhobm^{t} \cdot \n^{t})\circ \X^{t} = 0,
\end{align}
where the last equality follows because $\rhobm^{t}$ is tangential and $\n^{t}$ is normal. Hence $| \partial_{\rhobm^{0}}\X^{t} |$ is constant. Since $ \nabla_{\Gamma^0}\X^{t} = (\partial_{\rhobm^{0}}\X^{t})(\rhobm^{0})^{\top} $ by \eqref{t_grad},
\begin{align}
    | \nabla_{\Gamma^0} \X^{t} | = | \partial_{\rhobm^{0}}\X^{t} | = \text{constant},
\end{align}
for fixed $t$.
\end{proof}

For $0\le r,t\le T$, define
$\X^{r,t}:=\X^{t}\circ(\X^{r})^{-1}:\Gamma^{r}\to\Gamma^{t}$.
\begin{lemma}\label{lem_trans_grad}
Under the diffeomorphism and nondegeneracy assumptions above, any scalar- or vector-valued smooth function $f$ on $\Gamma^t$ satisfies
\begin{align}\label{idt_ptros_fXst}
\partial_{\rhobm^r} (f \circ \X^{r,t}) = |\partial_{\rhobm^r}\X^{r,t}|(\partial_{\rhobm^t}f)\circ\X^{r,t}.
\end{align}
\end{lemma}
\begin{proof}
Replace $\X^t$ in \eqref{rel_pt_rho_s} by $\X^{r}$ and let $g$ be defined on $\Gamma^r$. Then
\begin{align}\label{rel_pt_rho_s_Xs_f}
\partial_{\rhobm^0}(g\circ\X^r)=|\partial_{\rhobm^0}\X^r|(\partial_{\rhobm^r}g)\circ\X^r.
\end{align}
Applying \eqref{rel_pt_rho_s_Xs_f} to $f\circ\X^{r,t}$ gives
\begin{align}
\partial_{\rhobm^0}(f\circ\X^t)=\partial_{\rhobm^0}(f\circ\X^{r,t}\circ\X^r)
=|\partial_{\rhobm^0}\X^r|(\partial_{\rhobm^r}(f\circ\X^{r,t}))\circ\X^r. \label{pt_ro0_fXs_t_Xs}
\end{align}
Combining this identity with \eqref{rel_pt_rho_s} yields
\begin{align}
|\partial_{\rhobm^0}\X^t|(\partial_{\rhobm^t}f)\circ\X^t
=|\partial_{\rhobm^0}\X^r|(\partial_{\rhobm^r}(f\circ\X^{r,t}))\circ\X^r.
\end{align}
Composing with $(\X^r)^{-1}$ gives
\begin{align}\label{idt_ptros_Xf}
\partial_{\rhobm^r}(f\circ\X^{r,t})
=\frac{|\partial_{\rhobm^0}\X^t|}{|\partial_{\rhobm^0}\X^r|}\circ(\X^r)^{-1}(\partial_{\rhobm^t}f)\circ\X^{r,t}.
\end{align}
Applying \eqref{rel_pt_rho_s_Xs_f} to $\X^{r,t}$ also gives
\begin{align}
\partial_{\rhobm^r}\X^{r,t}=\frac{\partial_{\rhobm^0}\X^t}{|\partial_{\rhobm^0}\X^r|}\circ(\X^r)^{-1},
\end{align}
and therefore
\begin{align}\label{rel_pt_rho_s_X}
|\partial_{\rhobm^r}\X^{r,t}|=\frac{|\partial_{\rhobm^0}\X^t|}{|\partial_{\rhobm^0}\X^r|}\circ(\X^r)^{-1}.
\end{align}
Substitution of \eqref{rel_pt_rho_s_X} into \eqref{idt_ptros_Xf} proves the lemma.
\end{proof}

The harmonic pullback transforms the Laplace--Beltrami operator as follows.
\begin{lemma}\label{lem_trans_lap}
If the map $\X^{\theta}:\Gamma^0\to\Gamma^{\theta}$ is harmonic for every $\theta\in(0,T]$, then for any $0\le r\le t\le T$ and any smooth function $f:\Gamma^t\to\R^n$ (with $n=1$ or $2$, where for $n=2$ both the Laplacian and the composition act componentwise),
\begin{align}
\Delta_{\Gamma^r}(f\circ\X^{r,t})=|\nabla_{\Gamma^r}\X^{r,t}|^2(\Delta_{\Gamma^t}f)\circ\X^{r,t}.
\end{align}
\end{lemma}
\begin{proof}
Applying $\partial_{\rhobm^r}$ to \eqref{idt_ptros_fXst} gives
\begin{align}
\Delta_{\Gamma^r}(f\circ\X^{r,t})
&=\partial_{\rhobm^r\rhobm^r}(f\circ\X^{r,t}) \quad \text{by \eqref{lap_bel_t}} \notag\\
&=\partial_{\rhobm^r}\bigl(|\partial_{\rhobm^r}\X^{r,t}|(\partial_{\rhobm^t}f)\circ\X^{r,t}\bigr).
\end{align}
By \eqref{rel_pt_rho_s_X} and Lemma~\ref{lemma:isometric}, $|\partial_{\rhobm^r}\X^{r,t}|$ is constant for fixed $r,t$. Consequently,
\begin{align}\label{lap_bel_st}
\Delta_{\Gamma^r}(f\circ\X^{r,t})=|\partial_{\rhobm^r}\X^{r,t}|\partial_{\rhobm^r}\bigl((\partial_{\rhobm^t}f)\circ\X^{r,t}\bigr).
\end{align}
Applying Lemma~\ref{lem_trans_grad} to $\partial_{\rhobm^t}f$ yields
\begin{align}\label{pt_rho_s_pt_rho_t_fXst}
\partial_{\rhobm^r}\bigl((\partial_{\rhobm^t}f)\circ\X^{r,t}\bigr)
=|\partial_{\rhobm^r}\X^{r,t}|(\partial_{\rhobm^t\rhobm^t}f)\circ\X^{r,t}
=|\partial_{\rhobm^r}\X^{r,t}|(\Delta_{\Gamma^t}f)\circ\X^{r,t}.
\end{align}
Substituting \eqref{pt_rho_s_pt_rho_t_fXst} into \eqref{lap_bel_st} proves the result.
\end{proof}

For a closed curve, apply Lemma~\ref{lem_trans_lap} componentwise to $f=\id_{\Gamma^t}$ and use \eqref{eq:intro_mcf}.

\begin{proposition}\label{lem_mid_map}
If the map $\X^{\theta}:\Gamma^0\to\Gamma^{\theta}$ is harmonic for every $\theta\in(0,T]$ and $\Gamma^0$ is a closed curve, then for any $0\le r\le t\le T$ the pullback map $\X^{r,t}:\Gamma^r\to\Gamma^t$ is also harmonic and satisfies
\begin{subequations}
\begin{align}
\partial_t\X^{r,t}\cdot(\n^t\circ\X^{r,t})&=(\v^t\cdot\n^t)\circ\X^{r,t}\quad\mbox{on}\quad\Gamma^r,\label{eqn_Xst_a}\\
-\Delta_{\Gamma^r}\X^{r,t}&=|\nabla_{\Gamma^r}\X^{r,t}|^2(\kappa^t\n^t)\circ\X^{r,t}\quad\mbox{on}\quad\Gamma^r.\label{eqn_Xst_b}
\end{align}
\end{subequations}
In particular, $\lambda^t\circ\X^t=|\nabla_{\Gamma^0}\X^t|^2\kappa^t\circ\X^t$.
\end{proposition}
\begin{proof}
Equation~\eqref{eqn_Xst_a} follows from \eqref{pde-X-t0}. For \eqref{eqn_Xst_b}, apply Lemma~\ref{lem_trans_lap} componentwise to $ f = \id_{\Gamma^t} $:
\begin{align}
\Delta_{\Gamma^r}\X^{r,t}=|\nabla_{\Gamma^r}\X^{r,t}|^2(\Delta_{\Gamma^t}\id_{\Gamma^t})\circ\X^{r,t}=-|\nabla_{\Gamma^r}\X^{r,t}|^2(\kappa^t\n^t)\circ\X^{r,t}.
\end{align}
The last equality uses \eqref{eq:intro_mcf}. Taking $r=0$ in \eqref{eqn_Xst_b} and comparing it with \eqref{pde-kappa-t0} gives the multiplier identity.
\end{proof}

\subsection{Energy stability for curve-shortening flow}
Let $0=t_0<t_1<\cdots<t_M=T$ have step sizes $\tau_m=t_{m+1}-t_m$, and set
\begin{align*}
    \Gamma^{m} & := \Gamma^{t_m}, \\
    \X^{m,t} & := \X^{t_m , t}.
\end{align*}
For curve-shortening flow, $\v^t\cdot\n^t=-\kappa^t$. Proposition~\ref{lem_mid_map} gives the following continuous system on the time-node curve $\Gamma^m$ for $t\in(t_m,t_{m+1}]$:
\begin{subequations}\label{eqn_Xtmt}
    \begin{align}
\partial_t \X^{m,t} \cdot (\n^t \circ \X^{m,t}) & = -\kappa^t \circ \X^{m,t} \quad \text{on}\quad \Gamma^{m}, \\
- \Delta_{\Gamma^{m}} \X^{m,t} & = |\nabla_{\Gamma^{m}} \X^{m,t}|^2 (\kappa^{t} \n^{t}) \circ \X^{m,t}\quad  \text{on} \quad \Gamma^{m},\\
\X^{m,t_m}&=\id_{\Gamma^m}.
\end{align}
\end{subequations}
Multiplying the normal-velocity equation by $|\nabla_{\Gamma^{m}}\X^{m,t}|^2$ places the same metric weight in both relations and gives the BGN-type continuous formulation
\begin{subequations}\label{PDE-MCF-BGN-High}
\begin{align}
|\nabla_{\Gamma^{m}} \X^{m,t}|^2\partial_t \X^{m , t} \cdot (\n^t \circ \X^{m , t}) & = -|\nabla_{\Gamma^{m}} \X^{m ,t}|^2\kappa^t \circ \X^{m , t} \quad \text{on}\quad \Gamma^{m} \\
- \Delta_{\Gamma^{m}} \X^{m ,t} & = |\nabla_{\Gamma^{m}} \X^{m,t}|^2 (\kappa^{t} \n^{t}) \circ \X^{m ,t}\quad  \text{on} \quad \Gamma^{m} .
\end{align}
\end{subequations}

The weak formulation of \eqref{PDE-MCF-BGN-High} is: find a flow map $\X^{m,t}:\Gamma^{m}\to\Gamma^{t}$ with $\X^{m,t_m}=\id_{\Gamma^m}$ and a function $\kappa^t:\Gamma^t\to\R$ for $t\in(t_m,t_{m+1}]$ such that
\begin{subequations}\label{MCF-BGN-High-weak}
\begin{align}
& \int_{\Gamma^m}|\nabla_{\Gamma^m} \X^{m,t}|^2\partial_t \X^{m,t} \cdot  (\n^t \circ \X^{m,t}) \varphi  \, \d \Gamma^m = - \int_{\Gamma^{m}} |\nabla_{\Gamma^{m}} \X^{m,t}|^2  (\kappa^t \circ \X^{m,t}) \varphi \, \d \Gamma^m,
\\
& \int_{\Gamma^m} \nabla_{\Gamma^m} \X^{m,t} : \nabla_{\Gamma^m} \w  \, \d \Gamma^m = \int_{\Gamma^m} |\nabla_{\Gamma^m} \X^{m,t} |^2  [(\kappa^t \n^t )\circ \X^{m,t}]\cdot \w \, \d \Gamma^m
,
\end{align}
\end{subequations}
for all test functions $\varphi \in L^{2}(\Gamma^m)$ and $\w \in H^1(\Gamma^m)^2$.

\medskip
\noindent {\bf Curve-shortening stage equations and common RK--PFEM update}
\smallskip

{\em Step 1:}
For a given $\x^m$, let $\Gamma_h^m=\Gamma_h[\x^m]$ and set $\u_h^{m,0}=\id_{\Gamma_h^m}$. At $t=t_m+c_i\tau_m$, the functions $\u_h^{m,i}$, $\kappa_h^{m,i}$, and $F_h^{m,i}$ approximate $\X^{m,t}$, $\kappa^t\circ\X^{m,t}$, and $\partial_t\X^{m,t}$, respectively. For each $K\in\mathcal T_h^m$, let $\n_h^{m,i}|_K$ be the oriented unit normal of the image edge $\u_h^{m,i}(K)$, pulled back to $K$ by the element correspondence. Solve for $\{\u_h^{m,i},\kappa_h^{m,i},F_h^{m,i}\}_{i=1}^s\subset S_h[\x^m]^2\times S_h[\x^m]\times S_h[\x^m]^2$:
\begin{subequations}\label{MCF-BGN-High-fully}
\begin{align}
\label{MCF-scheme-stage}
& \u_h^{m,i} = \u_h^{m,0} + \tau_m \sum_{j=1}^{s} a_{ij} F_h^{m,j}, \qquad i = 1,\cdots ,s,\\
\label{MCF-scheme-Xi}
& \Big(|\nabla_{\Gamma_h^m}\u_h^{m,i}|^2 F_h^{m,i}\cdot \n_h^{m,i}, \, \varphi_h \Big)_{\Gamma_h^{m}}^h  = - \Big( |\nabla_{\Gamma_h^m}\u_h^{m,i}|^2 \kappa_h^{m,i} \, ,\varphi_h \Big)_{\Gamma_h^{m}}^h, \quad i = 1,\cdots ,s,\\
\label{MCF-scheme-Hi}
& \Big(|\nabla_{\Gamma_h^m}\u_h^{m,i}|^2 \kappa_h^{m,i} \n_h^{m,i}, \w_h \Big)_{\Gamma_h^{m}}^h  = \Big(\nabla_{\Gamma_h^{m}}  \u_h^{m,i}, \, \nabla_{\Gamma_h^{m}} \w_h \Big)_{\Gamma_h^{m}}, \quad i = 1,\cdots ,s,
\end{align}
\end{subequations}
for all $\varphi_h \in S_h[\x^{m}]$ and $\w_h \in S_h[\x^{m}]^2$.

\smallskip

{\em Step 2:} Compute $ \u_h^{m+1} \in S_h [\x^m]^2 $ from
\begin{align}\label{rk-pfem-update}
\u_h^{m+1} = \u_h^{m,0} + \tau_m \sum_{i=1}^{s} b_i F_h^{m,i} .
\end{align}
\smallskip

{\em Step 3:} Set $\x^{m+1}=\u_h^{m+1}(\x^m)$ and $\Gamma_h^{m+1}:=\Gamma_h[\x^{m+1}]$. The stage relation \eqref{MCF-scheme-stage}, endpoint update \eqref{rk-pfem-update}, and mesh push-forward are common to all four flow-specific stage systems below.

All zero-order products in the stage equations are mass lumped, as in
the classical BGN discretization
\cite{barrett2007parametric,bao2021structure,zhao2021energy}, whereas
gradient terms use standard elementwise integration. Each fixed-point
iteration freezes the metric weights and intermediate normals, reducing
the nonlinear stage system to a sparse linear BGN system.

\begin{lemma}[Per-stage energy decay for curve-shortening flow]\label{lem:perstage-mcf}
For each stage $i = 1,\dots,s$, the solution of the nonlinear scheme \eqref{MCF-BGN-High-fully} satisfies
\begin{align}\label{MCF-perstage}
\big(\nabla_{\Gamma_h^{m}}\u_h^{m,i},\, \nabla_{\Gamma_h^{m}}F_h^{m,i}\big)_{\Gamma_h^{m}}
= - \Big(|\nabla_{\Gamma_h^m}\u_h^{m,i}|^2\kappa_h^{m,i},\,\kappa_h^{m,i}\Big)^h_{\Gamma_h^m} \le 0.
\end{align}
\end{lemma}

\begin{proof}
Choose $\varphi_h =  \kappa_h^{m,i}$ in \eqref{MCF-scheme-Xi} to obtain
\begin{align}\label{MCF-energy-p1}
\big(|\nabla_{\Gamma_h^m}\u_h^{m,i}|^2 F_h^{m,i}\cdot \n_h^{m,i}, \, \kappa_h^{m,i} \big)_{\Gamma_h^{m}}^h
= - \Big(|\nabla_{\Gamma_h^m}\u_h^{m,i}|^2\kappa_h^{m,i},\,\kappa_h^{m,i}\Big)^h_{\Gamma_h^m} \le 0.
\end{align}
Choosing $\w_h = F_h^{m,i}$ in \eqref{MCF-scheme-Hi} gives
\begin{align}\label{MCF-energy-p2}
\big(\nabla_{\Gamma_h^{m}} \u_h^{m,i}, \, \nabla_{\Gamma_h^{m}} F_h^{m,i}\big)_{\Gamma_h^{m}}
= \Big(|\nabla_{\Gamma_h^m}\u_h^{m,i}|^2 \kappa_h^{m,i} \n_h^{m,i}, \, F_h^{m,i} \Big)_{\Gamma_h^{m}}^h.
\end{align}
By symmetry of the mass-lumped inner product and
$(\kappa_h^{m,i}\n_h^{m,i})\cdot F_h^{m,i}
=\kappa_h^{m,i}(F_h^{m,i}\cdot\n_h^{m,i})$,
equations \eqref{MCF-energy-p1} and \eqref{MCF-energy-p2}
yield \eqref{MCF-perstage}.
\end{proof}

\begin{theorem}[Energy decay]\label{thm:mcf_energy_decay_curve}
Let the RK method be algebraically stable in the sense of \eqref{alg-stable}. Fix $m\in\{0,\ldots,M-1\}$ and $\tau_m>0$. Assume that \eqref{MCF-BGN-High-fully} admits an exact solution whose intermediate polygonal curves are nondegenerate. Then
\begin{align}
|\Gamma_h^{m+1}|\le |\Gamma_h^m|.
\end{align}
\end{theorem}

\begin{proof}%
The gradient bilinear form is positive definite on $V=S_h[\x^m]^2/\R^2$. Apply Lemma~\ref{lem:B-stability} in this quotient Hilbert space with
$\langle\cdot,\cdot\rangle=(\nabla_{\Gamma_h^m}\cdot,\nabla_{\Gamma_h^m}\cdot)_{\Gamma_h^m}$,
with stage values $\u_h^{m,i}$ and stage velocities $F_h^{m,i}$. Lemma~\ref{lem:perstage-mcf} gives $\langle\u_h^{m,i},F_h^{m,i}\rangle\le0$, and hence
\begin{align}\label{MCF-Bstab}
&\frac{1}{2}\Big(\|\nabla_{\Gamma_h^m}\u_h^{m+1}\|^2_{L^2(\Gamma_h^m)}
-\|\nabla_{\Gamma_h^m}\u_h^{m,0}\|^2_{L^2(\Gamma_h^m)}\Big)\notag\\
&\quad=\tau_m\sum_{i=1}^{s} b_i
\big(\nabla_{\Gamma_h^{m}}\u_h^{m,i},\nabla_{\Gamma_h^{m}}F_h^{m,i}\big)_{\Gamma_h^{m}}\notag\\
&\qquad-\frac{\tau_m^2}{2}\sum_{i,j=1}^{s}\mathcal M_{ij}
\big(\nabla_{\Gamma_h^{m}}F_h^{m,i},\nabla_{\Gamma_h^{m}}F_h^{m,j}\big)_{\Gamma_h^{m}}\le0,
\end{align}
where the first sum is nonpositive by $b_i\ge0$ and \eqref{MCF-perstage}. The second sum is nonpositive after its displayed minus sign because $\mathcal M\succeq0$ and the matrix of gradient inner products is a Gram matrix.

Since $\u_h^{m,0}=\id_{\Gamma_h^m}$ and $|\nabla_{\Gamma_h^m}\id_{\Gamma_h^m}|=1$, applying $2x\le x^2+1$ to $x=|\nabla_{\Gamma_h^m}\u_h^{m+1}|\ge0$ gives
\begin{align} \label{energy-proof-4}
|\Gamma_h^{m+1}|-  |\Gamma_h^{m}|
& = \int_{\Gamma_h^m}\Big(| \nabla_{\Gamma_h^m}\u_h^{m+1} |-1\Big)\\
& \le \frac{1}{2} \int_{\Gamma_h^{m}} |\nabla_{\Gamma_h^{m}} \u_h^{m+1}|^2 - \frac{1}{2} \int_{\Gamma_h^{m}} |\nabla_{\Gamma_h^{m}} \id_{\Gamma_h^m}|^2\\
& = \frac{1}{2}\Big(\|\nabla_{\Gamma_h^m}\u_h^{m+1}\|^2_{L^2(\Gamma_h^m)} - \|\nabla_{\Gamma_h^m}\u_h^{m,0}\|^2_{L^2(\Gamma_h^m)}\Big) \le 0,
\end{align}
This proves the stated energy decay.
\end{proof}

\subsection{Energy stability for curve diffusion}
For curve diffusion, $\v^t\cdot\n^t=\Delta_{\Gamma^t}\kappa^t$. On each time slab, Proposition~\ref{lem_mid_map} supplies the metric-weighted curvature equation on $\Gamma^m$, while Lemma~\ref{lem_trans_lap} gives
\begin{align}\label{eqn_Dekat_Dem}
\Delta_{\Gamma^m}(\kappa^{t} \circ \X^{m,t}) = | \nabla_{\Gamma^m}\X^{m,t} |^2  (\Delta_{\Gamma^{t}} \kappa^{t}) \circ \X^{m,t}.
\end{align}
Combining this identity with the pulled-back normal-velocity equation yields the time-slab BGN pair
\begin{subequations}\label{eqn_Xtmt_sd_alt}
    \begin{align}
| \nabla_{\Gamma^m}\X^{m,t} |^2\partial_t \X^{m,t} \cdot (\n^t \circ \X^{m , t}) & = \Delta_{\Gamma^m}(\kappa^{t} \circ \X^{m,t}) \quad \text{on}\quad \Gamma^{m} \\
- \Delta_{\Gamma^{m}} \X^{m ,t} & = |\nabla_{\Gamma^{m}} \X^{m,t}|^2 (\kappa^{t} \n^{t}) \circ \X^{m ,t}\quad  \text{on} \quad \Gamma^{m} .
    \end{align}
\end{subequations}
The weak formulation of \eqref{eqn_Xtmt_sd_alt} is: find a flow map $\X^{m,t}: \Gamma^{m} \to  \Gamma^{t}$ with $\X^{m,t_m}=\id_{\Gamma^m}$ and a function $\kappa^t: \Gamma^t \to  \R$ for $ t \in (t_m, t_{m+1}] $ such that
\begin{subequations}\label{sd-BGN-High-weak}
\begin{align}
& \int_{\Gamma^m}|\nabla_{\Gamma^m} \X^{m,t}|^2\partial_t \X^{m,t} \cdot  (\n^t \circ \X^{m,t}) \varphi  \, \d \Gamma^m = - \int_{\Gamma^{m}} \nabla_{\Gamma^m}(\kappa^t \circ \X^{m,t})\cdot \nabla_{\Gamma^m} \varphi \, \d \Gamma^m,
\\
& \int_{\Gamma^m} \nabla_{\Gamma^m} \X^{m,t} : \nabla_{\Gamma^m} \w  \, \d \Gamma^m = \int_{\Gamma^m} |\nabla_{\Gamma^m} \X^{m,t} |^2  [(\kappa^t \n^t )\circ \X^{m,t}]\cdot \w \, \d \Gamma^m
,
\end{align}
\end{subequations}
for all test functions $\varphi \in H^{1}(\Gamma^m)$ and $\w \in H^1(\Gamma^m)^2$.

Let $S_h[\x^m]$ be the finite element space on the polygonal curve $\Gamma_h^m=\Gamma_h[\x^m]$ approximating $\Gamma^m$, and let $(c,A,b)$ be an $s$-stage algebraically stable RK method.

\medskip
\noindent {\bf Flow-specific stage equations}
\smallskip

Using the stage notation and elementwise pulled-back normals defined above, set $\u_h^{m,0}=\id_{\Gamma_h^m}$ and solve for $\{\u_h^{m,i},\kappa_h^{m,i},F_h^{m,i}\}_{i=1}^s\subset S_h[\x^m]^2\times S_h[\x^m]\times S_h[\x^m]^2$:
\begin{subequations}\label{sd-BGN-High-fully}
\begin{align}
& \u_h^{m,i} = \u_h^{m,0} + \tau_m \sum_{j=1}^{s} a_{ij} F_h^{m,j}, \qquad i = 1,\cdots ,s,\\
\label{sd-scheme-Xi}
& \Big(|\nabla_{\Gamma_h^m}\u_h^{m,i}|^2 F_h^{m,i}\cdot \n_h^{m,i}, \, \varphi_h \Big)_{\Gamma_h^{m}}^h  = - \Big( \nabla_{\Gamma_h^m}  \kappa_h^{m,i} \, ,\nabla_{\Gamma_h^m}\varphi_h \Big)_{\Gamma_h^{m}}, \quad i = 1,\cdots ,s,\\
\label{sd-scheme-Hi}
& \Big(|\nabla_{\Gamma_h^m}\u_h^{m,i}|^2   \kappa_h^{m,i} \n_h^{m,i}, \w_h \Big)_{\Gamma_h^{m}}^h  = \Big(\nabla_{\Gamma_h^{m}}  \u_h^{m,i}, \, \nabla_{\Gamma_h^{m}} \w_h \Big)_{\Gamma_h^{m}}, \quad i = 1,\cdots ,s,
\end{align}
\end{subequations}
for all test functions $\varphi_h \in S_h[\x^{m}]$ and $\w_h \in S_h[\x^{m}]^2$.

The time-node value and the next polygonal curve are then obtained from the common endpoint update \eqref{rk-pfem-update} and mesh push-forward.

\begin{lemma}[Per-stage energy decay for curve diffusion]\label{lem:perstage-sd}
For each stage $i = 1,\dots,s$, the solution of the nonlinear scheme \eqref{sd-BGN-High-fully} satisfies
\begin{align}\label{sd-perstage}
\big(\nabla_{\Gamma_h^{m}}\u_h^{m,i},\, \nabla_{\Gamma_h^{m}}F_h^{m,i}\big)_{\Gamma_h^{m}}
= - \int_{\Gamma_h^{m}} |\nabla_{\Gamma_h^{m}} \kappa_h^{m,i}|^2 \le 0.
\end{align}
\end{lemma}

\begin{proof}
Choose $\varphi_h =  \kappa_h^{m,i}$ in \eqref{sd-scheme-Xi}:
\begin{align}\label{sd-energy-p1}
\big(|\nabla_{\Gamma_h^m}\u_h^{m,i}|^2 F_h^{m,i}\cdot \n_h^{m,i}, \, \kappa_h^{m,i} \big)_{\Gamma_h^{m}}^h
= - \int_{\Gamma_h^{m}}  |\nabla_{\Gamma_h^{m}} \kappa_h^{m,i}|^2 \le 0.
\end{align}
Taking $\w_h = F_h^{m,i}$ in \eqref{sd-scheme-Hi} gives
\begin{align}\label{sd-energy-p2}
\big(\nabla_{\Gamma_h^{m}} \u_h^{m,i}, \, \nabla_{\Gamma_h^{m}} F_h^{m,i}\big)_{\Gamma_h^{m}}
= \Big(|\nabla_{\Gamma_h^m}\u_h^{m,i}|^2   \kappa_h^{m,i} \n_h^{m,i}, \, F_h^{m,i} \Big)_{\Gamma_h^{m}}^h.
\end{align}
The two mass-lumped products coincide by symmetry and $(\kappa_h^{m,i}\n_h^{m,i})\cdot F_h^{m,i}=\kappa_h^{m,i}(F_h^{m,i}\cdot\n_h^{m,i})$, with the same weight $|\nabla_{\Gamma_h^m}\u_h^{m,i}|^2$ and normal $\n_h^{m,i}$ in both stage equations. Hence \eqref{sd-perstage} follows.
\end{proof}

\begin{theorem}[Energy decay]\label{thm:sd_energy_decay_curve}
Let the RK method be algebraically stable as in Theorem~\ref{thm:mcf_energy_decay_curve}. Fix $m\in\{0,\ldots,M-1\}$ and $\tau_m>0$. Assume that \eqref{sd-BGN-High-fully} admits an exact solution whose intermediate polygonal curves are nondegenerate. Then
\begin{align}
|\Gamma_h^{m+1}|\le|\Gamma_h^m|.
\end{align}
\end{theorem}

\begin{proof}%
Applying \eqref{B-stability-identity} on $S_h[\x^m]^2/\R^2$ with the gradient inner product and using Lemma~\ref{lem:perstage-sd} gives
\begin{align}\label{sd-Bstab}
&\frac{1}{2}\Big(\|\nabla_{\Gamma_h^m}\u_h^{m+1}\|^2_{L^2(\Gamma_h^m)}
-\|\nabla_{\Gamma_h^m}\u_h^{m,0}\|^2_{L^2(\Gamma_h^m)}\Big)\notag\\
&\quad=\tau_m\sum_{i=1}^{s}b_i
(\nabla_{\Gamma_h^{m}}\u_h^{m,i},\nabla_{\Gamma_h^{m}}F_h^{m,i})_{\Gamma_h^{m}}\notag\\
&\qquad-\frac{\tau_m^2}{2}\sum_{i,j=1}^{s}\mathcal M_{ij}
(\nabla_{\Gamma_h^{m}}F_h^{m,i},\nabla_{\Gamma_h^{m}}F_h^{m,j})_{\Gamma_h^{m}}\le0,
\end{align}
The geometric inequality \eqref{energy-proof-4} gives the stated energy decay.
\end{proof}

\section{RK-based PFEMs for surface flows}\label{Se:4}
Throughout this section, $\Gamma^t\subset\R^3$ is a closed genus-$0$ surface. The harmonic map into each intermediate-time surface is written on the current time-node surface $\Gamma^m$. Conformality then produces the weighted mean-curvature relation discretized by the Runge--Kutta stages.

\subsection{Time-slab conformal formulation for surfaces}
For a harmonic map between closed genus-$0$ surfaces, the Hopf differential is holomorphic and therefore vanishes. Hence an orientation-preserving harmonic diffeomorphism is conformal \cite{heleinWood2008harmonic}. We use this conformality to express the continuous harmonic-map multiplier through mean curvature on every time slab.

Let $\Gamma^t$ move with velocity $\v^t$, let $0=t_0<t_1<\cdots<t_M=T$, set $\Gamma^m:=\Gamma^{t_m}$, and define $I_m:=(t_m,t_{m+1}]$. On $I_m$, let $\X^{m,t}:\Gamma^m\to\Gamma^t$ be an orientation-preserving harmonic diffeomorphism that is continuously differentiable in time and twice continuously differentiable in space. Then

\begin{subequations}\label{pde-kappa-t0_surf}
	\begin{align}
		\partial_t \X^{m , t} \cdot (\n^t \circ \X^{m , t}) & =  (\v^t \cdot \n^t) \circ \X^{m , t} \quad \text{on}\quad \Gamma^{m} \\
		- \Delta_{\Gamma^{m}} \X^{m,t} & = (\lambda^{m,t} \n^t) \circ \X^{m,t} \quad  \text{on} \quad \Gamma^{m},\label{pde-kappa-t0_surf_b}\\
        \X^{m,t_m}&=\id_{\Gamma^m},
	\end{align}
\end{subequations}
where $\lambda^{m,t}:\Gamma^t\to\R$ is the time-slab Lagrange multiplier and $\n^t$ is the unit outward normal to $\Gamma^t$. At $t=t_m+c_i\tau_m$, the target of $\X^{m,t}$ is the $i$th intermediate-time surface, while both equations remain posed on $\Gamma^m$.

Let $g_m$ and $g_t$ denote the induced metrics on $\Gamma^m$ and $\Gamma^t$. Conformality gives
\begin{align}\label{surface-conformal-factor}
(\X^{m,t})^*g_t
=\frac12|\nabla_{\Gamma^m}\X^{m,t}|^2g_m.
\end{align}
For every smooth scalar function $f$ on $\Gamma^t$, the two-dimensional conformal transformation law gives
\begin{align}\label{surface-laplacian-pullback}
\Delta_{\Gamma^m}(f\circ\X^{m,t})
=\frac12|\nabla_{\Gamma^m}\X^{m,t}|^2
(\Delta_{\Gamma^t}f)\circ\X^{m,t}.
\end{align}
\begin{proposition}\label{surf_X_harm}
Let $\Gamma^m$ and $\Gamma^t$ be closed genus-$0$ surfaces, and let $\X^{m,t}:\Gamma^m\to\Gamma^t$ be the orientation-preserving harmonic diffeomorphism specified above. Then
\begin{align}\label{pde-mcf-H-t0-surf}
-\Delta_{\Gamma^m}\X^{m,t}
=\frac12|\nabla_{\Gamma^m}\X^{m,t}|^2
  (H^t\n^t)\circ\X^{m,t}
\qquad\text{on }\Gamma^m,
\end{align}
where $\Delta_{\Gamma^t}\id_{\Gamma^t}=-H^t\n^t$. In particular, the multiplier in \eqref{pde-kappa-t0_surf_b} is
\begin{align}\label{eq:surface-multiplier}
\lambda^{m,t}
=\frac12\Bigl(|\nabla_{\Gamma^m}\X^{m,t}|^2
  \circ(\X^{m,t})^{-1}\Bigr)H^t
\qquad\text{on }\Gamma^t.
\end{align}
\end{proposition}

\begin{proof}
Apply \eqref{surface-laplacian-pullback} componentwise to
$f=\id_{\Gamma^t}$ and use
$\Delta_{\Gamma^t}\id_{\Gamma^t}=-H^t\n^t$ to obtain
\begin{align*}
\Delta_{\Gamma^m}\X^{m,t}
&=\frac12|\nabla_{\Gamma^m}\X^{m,t}|^2
  (\Delta_{\Gamma^t}\id_{\Gamma^t})\circ\X^{m,t}\\
&=-\frac12|\nabla_{\Gamma^m}\X^{m,t}|^2
  (H^t\n^t)\circ\X^{m,t}.
\end{align*}
This proves \eqref{pde-mcf-H-t0-surf}, and comparison with
\eqref{pde-kappa-t0_surf_b} gives \eqref{eq:surface-multiplier}.
\end{proof}

\subsection{Energy stability for mean curvature flow}

For mean curvature flow, the surface has normal velocity $-H^t\n^t$. Proposition~\ref{surf_X_harm} gives the following system on $\Gamma^m$ for $t\in I_m$:
\begin{subequations}\label{eqn_mcf_surf_harm}
\begin{align}
    \partial_t \X^{m,t} \cdot (\n^{t}\circ \X^{m,t})  & =  -H^{t}\circ \X^{m,t}\quad  \mbox{on} \quad \Gamma^m,\\
- \Delta_{\Gamma^m} \X^{m,t} & = \frac{1}{2}|\nabla_{\Gamma^{m}} \X^{m,t}|^2 (H^{t} \n^{t}) \circ \X^{m,t}\quad  \mbox{on} \quad \Gamma^m .
\end{align} 
\end{subequations}
Multiplying the normal-velocity equation by $|\nabla_{\Gamma^m}\X^{m,t}|^2$ produces the weighted formulation used for the energy estimate:
\begin{subequations}\label{PDE-MCF-BGN-High_surf}
\begin{align}
		|\nabla_{\Gamma^{m}} \X^{m,t}|^2\partial_t \X^{m , t} \cdot (\n^t \circ \X^{m , t}) & = -|\nabla_{\Gamma^{m}} \X^{m ,t}|^2 H^t \circ \X^{m , t} \quad \text{on}\quad \Gamma^{m} \\
- \Delta_{\Gamma^{m}} \X^{m ,t} & = \frac{1}{2} |\nabla_{\Gamma^{m}} \X^{m,t}|^2 (H^{t} \n^{t}) \circ \X^{m ,t}\quad  \text{on} \quad \Gamma^{m} .
\end{align}
\end{subequations}

The weak formulation is: find a flow map $\X^{m,t}:\Gamma^m\to\Gamma^t$ and a function $H^t:\Gamma^t\to\R$ for $t\in(t_m,t_{m+1}]$ such that
\begin{subequations}\label{MCF-BGN-High-weak_surf}
\begin{align} 
& \int_{\Gamma^m}|\nabla_{\Gamma^m} \X^{m,t}|^2\partial_t \X^{m,t} \cdot  (\n^t \circ \X^{m,t}) \varphi  \, \d \Gamma^m = - \int_{\Gamma^{m}} |\nabla_{\Gamma^{m}} \X^{m,t}|^2  (H^t \circ \X^{m,t}) \varphi \, \d \Gamma^m,
\\
& \int_{\Gamma^m} \nabla_{\Gamma^m} \X^{m,t} : \nabla_{\Gamma^m} \w  \, \d \Gamma^m = \frac{1}{2} \int_{\Gamma^m} |\nabla_{\Gamma^m} \X^{m,t} |^2  [(H^t \n^t )\circ \X^{m,t}]\cdot \w \, \d \Gamma^m 
,
\end{align}
\end{subequations}
for all test functions $\varphi \in L^{2}(\Gamma^m)$ and $\w \in H^1(\Gamma^m)^3$.

Let $(c,A,b)$ be an $s$-stage algebraically stable RK method, and let $S_h[\x^m]$ be the finite element space on the polygonal surface $\Gamma_h^m=\Gamma_h[\x^m]$ approximating $\Gamma^m$. Runge--Kutta time discretization together with the spatial finite element discretization and BGN mass lumping gives the following stage system.

\medskip
\noindent {\bf Flow-specific stage equations}
\smallskip

Set $\u_h^{m,0}=\id_{\Gamma_h^m}$. At $t=t_m+c_i\tau_m$, the functions $\u_h^{m,i}$, $H_h^{m,i}$, and $F_h^{m,i}$ approximate $\X^{m,t}$, $H^t\circ\X^{m,t}$, and $\partial_t\X^{m,t}$, respectively. For each $K\in\mathcal T_h^m$, let $\n_h^{m,i}|_K$ be the outward unit normal of the image triangle $\u_h^{m,i}(K)$, pulled back to $K$ by the element correspondence. Solve for $\{\u_h^{m,i},H_h^{m,i},F_h^{m,i}\}_{i=1}^s\subset S_h[\x^m]^3\times S_h[\x^m]\times S_h[\x^m]^3$:
\begin{subequations}\label{MCF-BGN-High-fully_surf}
\begin{align}
& \u_h^{m,i} = \u_h^{m,0} + \tau_m \sum_{j=1}^{s} a_{ij} F_h^{m,j}, \qquad i = 1,\cdots ,s,\\
\label{MCF-scheme-Xi_surf}
& \Big(|\nabla_{\Gamma_h^m}\u_h^{m,i}|^2 F_h^{m,i}\cdot \n_h^{m,i}, \, \varphi_h \Big)_{\Gamma_h^{m}}^h  = - \Big(|\nabla_{\Gamma_h^m}\u_h^{m,i}|^2   H_h^{m,i} \, ,\varphi_h  \Big)_{\Gamma_h^{m}}^h, \quad i = 1,\cdots ,s,\\
\label{MCF-scheme-Hi_surf}
& \frac{1}{2}\Big(|\nabla_{\Gamma_h^m}\u_h^{m,i}|^2   H_h^{m,i} \n_h^{m,i}, \w_h \Big)_{\Gamma_h^{m}}^h  =  \Big(\nabla_{\Gamma_h^{m}}  \u_h^{m,i}, \, \nabla_{\Gamma_h^{m}} \w_h \Big)_{\Gamma_h^{m}}, \quad i = 1,\cdots ,s,
\end{align}
\end{subequations}
for all test functions $\varphi_h \in S_h[\x^{m}]$ and $\w_h \in S_h[\x^{m}]^3$.

The time-node value and the next triangulated surface are then obtained from the common endpoint update \eqref{rk-pfem-update} and mesh push-forward.

\begin{lemma}[Per-stage energy decay for mean curvature flow of surfaces]\label{lem:perstage-surf}
For each stage $i=1,\dots,s$, any solution of the nonlinear scheme \eqref{MCF-BGN-High-fully_surf} satisfies
\begin{align}\label{MCF-perstage-surf}
\big(\nabla_{\Gamma_h^{m}}\u_h^{m,i},\, \nabla_{\Gamma_h^{m}}F_h^{m,i}\big)_{\Gamma_h^{m}}
= - \frac{1}{2}\Big(|\nabla_{\Gamma_h^m}\u_h^{m,i}|^2H_h^{m,i},\,H_h^{m,i}\Big)^h_{\Gamma_h^m} \le 0.
\end{align}
\end{lemma}

\begin{proof}
Test \eqref{MCF-scheme-Xi_surf} with $\varphi_h=H_h^{m,i}$ to obtain
\begin{align}\label{MCF-energy-p1_surf}
\big(|\nabla_{\Gamma_h^m}\u_h^{m,i}|^2 F_h^{m,i}\cdot \n_h^{m,i}, \, H_h^{m,i} \big)_{\Gamma_h^{m}}^h
= - \Big(|\nabla_{\Gamma_h^m}\u_h^{m,i}|^2H_h^{m,i},\,H_h^{m,i}\Big)^h_{\Gamma_h^m} \le 0.
\end{align}
Test \eqref{MCF-scheme-Hi_surf} with $\w_h=F_h^{m,i}$ to obtain
\begin{align}\label{MCF-energy-p2_surf}
\big(\nabla_{\Gamma_h^{m}} \u_h^{m,i}, \, \nabla_{\Gamma_h^{m}} F_h^{m,i}\big)_{\Gamma_h^{m}}
= \frac{1}{2}\Big(|\nabla_{\Gamma_h^m}\u_h^{m,i}|^2   H_h^{m,i} \n_h^{m,i}, \, F_h^{m,i} \Big)_{\Gamma_h^{m}}^h.
\end{align}
The mass-lumped product is symmetric, and $(H_h^{m,i}\n_h^{m,i})\cdot F_h^{m,i}=H_h^{m,i}(F_h^{m,i}\cdot\n_h^{m,i})$. Because the same weight and normal occur in both stage equations, the right-hand side of \eqref{MCF-energy-p2_surf} is one half of the left-hand side of \eqref{MCF-energy-p1_surf}. This proves \eqref{MCF-perstage-surf}. \end{proof}

\begin{theorem}[Energy decay]\label{thm:mcf_energy_decay_surf}
Let the RK method be algebraically stable as in Theorem~\ref{thm:mcf_energy_decay_curve}. Fix $m\in\{0,\ldots,M-1\}$ and $\tau_m>0$. Assume that \eqref{MCF-BGN-High-fully_surf} admits an exact solution whose intermediate triangulated surfaces are nondegenerate. Then
\begin{align}
|\Gamma_h^{m+1}|\le|\Gamma_h^m|.
\end{align}
\end{theorem}

\begin{proof}%
Apply \eqref{B-stability-identity} in the quotient Hilbert space $S_h[\x^m]^3/\R^3$ equipped with the gradient inner product. Lemma~\ref{lem:perstage-surf} then gives
\begin{align}\label{MCF-Bstab-surf}
&\frac{1}{2}\Big(\|\nabla_{\Gamma_h^m}\u_h^{m+1}\|^2_{L^2(\Gamma_h^m)}
-\|\nabla_{\Gamma_h^m}\u_h^{m,0}\|^2_{L^2(\Gamma_h^m)}\Big)\notag\\
&\quad=\tau_m\sum_{i=1}^{s}b_i
(\nabla_{\Gamma_h^{m}}\u_h^{m,i},\nabla_{\Gamma_h^{m}}F_h^{m,i})_{\Gamma_h^{m}}\notag\\
&\qquad-\frac{\tau_m^2}{2}\sum_{i,j=1}^{s}\mathcal M_{ij}
(\nabla_{\Gamma_h^{m}}F_h^{m,i},\nabla_{\Gamma_h^{m}}F_h^{m,j})_{\Gamma_h^{m}}\le0.
\end{align}

To convert the gradient estimate into area decay, let $\{\partial_{\rhobm_1^m},\partial_{\rhobm_2^m}\}$ be an orthonormal tangent frame on each planar face of $\Gamma_h^m$. The piecewise-linear image under $\u_h^{m+1}$ satisfies
\begin{align} \label{energy-proof-4_surf}
	|\Gamma_h^{m+1}|-  |\Gamma_h^{m}|
	& = \int_{\Gamma_h^m}\Big(| \partial_{\rhobm_1^m}\u_h^{m+1} \times \partial_{\rhobm_2^m}\u_h^{m+1} | - 1\Big)\\
	& \le \int_{\Gamma_h^m}\Big(| \partial_{\rhobm_1^m}\u_h^{m+1} |\,| \partial_{\rhobm_2^m}\u_h^{m+1} | - 1\Big)\\
	& \le \frac{1}{2}\int_{\Gamma_h^{m}} \Big(| \partial_{\rhobm_1^m}\u_h^{m+1} |^2 + | \partial_{\rhobm_2^m}\u_h^{m+1} |^2\Big)  - \frac{1}{2} \int_{\Gamma_h^{m}} |\nabla_{\Gamma_h^{m}} {\id}_{\Gamma_h^m}|^2\\
	& = \frac{1}{2} \int_{\Gamma_h^{m}} |\nabla_{\Gamma_h^{m}} \u_h^{m+1}|^2 - \frac{1}{2} \int_{\Gamma_h^{m}} |\nabla_{\Gamma_h^{m}} {\id}_{\Gamma_h^m}|^2\\
	& = \frac{1}{2}\Big(\|\nabla_{\Gamma_h^m}\u_h^{m+1}\|^2_{L^2(\Gamma_h^m)} - \|\nabla_{\Gamma_h^m}\u_h^{m,0}\|^2_{L^2(\Gamma_h^m)}\Big) \le 0.
\end{align}
The first inequality follows from $|a\times b|\le|a||b|$, and the second from $2xy\le x^2+y^2$. Finally, $|\nabla_{\Gamma_h^m}\id_{\Gamma_h^m}|^2=2$, so $\frac12\int|\nabla\id|^2=|\Gamma_h^m|$. Combining this identity with \eqref{MCF-Bstab-surf} proves the stated energy decay.
\end{proof}

\subsection{Energy stability for surface diffusion}

Surface diffusion has normal velocity $\Delta_{\Gamma^t}H^t$:
\begin{align}
\partial_t\X^{m,t}\cdot(\n^t\circ\X^{m,t}) = (\Delta_{\Gamma^t}H^t)\circ\X^{m,t},
\end{align}
where $\Delta_{\Gamma^t}$ is the Laplace--Beltrami operator on the evolving surface $\Gamma^t$ and $H^t$ its (scalar) mean curvature.

Applying the conformal transformation law \eqref{surface-laplacian-pullback} with $f=H^t$ gives
\begin{align}\label{eqn_Dekat_Dem_surf}
\Delta_{\Gamma^m}(H^t\circ\X^{m,t})
= \frac12|\nabla_{\Gamma^m}\X^{m,t}|^2
  (\Delta_{\Gamma^t}H^t)\circ\X^{m,t}.
\end{align}
Combining \eqref{eqn_Dekat_Dem_surf} with the normal-velocity equation yields
\begin{align}
\frac{1}{2}|\nabla_{\Gamma^m}\X^{m,t}|^2\,\partial_t\X^{m,t}\cdot(\n^t\circ\X^{m,t}) = \Delta_{\Gamma^m}(H^t\circ\X^{m,t}).
\end{align}
The corresponding weak formulation is to find $\X^{m,t}:\Gamma^m\to\Gamma^t$ with $\X^{m,t_m}=\id_{\Gamma^m}$ and $H^t:\Gamma^t\to\R$ such that
\begin{subequations}\label{sd-BGN-High-weak_surf}
\begin{align}
& \int_{\Gamma^m}\frac{1}{2}|\nabla_{\Gamma^m}\X^{m,t}|^2\partial_t\X^{m,t}\cdot(\n^t\circ\X^{m,t})\varphi\,\d\Gamma^m = -\int_{\Gamma^m}\nabla_{\Gamma^m}(H^t\circ\X^{m,t})\cdot\nabla_{\Gamma^m}\varphi\,\d\Gamma^m,\\
& \int_{\Gamma^m}\frac{1}{2}|\nabla_{\Gamma^m}\X^{m,t}|^2[(H^t\n^t)\circ\X^{m,t}]\cdot\w\,\d\Gamma^m = \int_{\Gamma^m}\nabla_{\Gamma^m}\X^{m,t}:\nabla_{\Gamma^m}\w\,\d\Gamma^m,
\end{align}
\end{subequations}
for all $\varphi\in H^1(\Gamma^m)$ and $\w\in H^1(\Gamma^m)^3$. The second equation is the weak mean-curvature relation from Proposition~\ref{surf_X_harm}. Evaluating this system at the Runge--Kutta internal times and applying the mass-lumped parametric finite element discretization gives the stage equations below.

\medskip
\noindent{\bf Flow-specific stage equations}
\smallskip

Using the stage notation and elementwise pulled-back normals defined above, set $\u_h^{m,0}=\id_{\Gamma_h^m}$ and solve for $\{\u_h^{m,i},H_h^{m,i},F_h^{m,i}\}_{i=1}^s\subset S_h[\x^m]^3\times S_h[\x^m]\times S_h[\x^m]^3$:
\begin{subequations}\label{sd-BGN-High-fully_surf}
\begin{align}
& \u_h^{m,i} = \u_h^{m,0} + \tau_m\sum_{j=1}^s a_{ij}F_h^{m,j}, \qquad i=1,\dots,s,\\
\label{sd-scheme-Xi_surf}
& \Big(\frac{1}{2}|\nabla_{\Gamma_h^m}\u_h^{m,i}|^2 F_h^{m,i}\cdot\n_h^{m,i},\,\varphi_h\Big)_{\Gamma_h^m}^h = -\Big(\nabla_{\Gamma_h^m}H_h^{m,i},\,\nabla_{\Gamma_h^m}\varphi_h\Big)_{\Gamma_h^m}, \quad i=1,\dots,s,\\
\label{sd-scheme-Hi_surf}
& \Big(\frac{1}{2}|\nabla_{\Gamma_h^m}\u_h^{m,i}|^2 H_h^{m,i}\n_h^{m,i},\,\w_h\Big)_{\Gamma_h^m}^h = \Big(\nabla_{\Gamma_h^m}\u_h^{m,i},\,\nabla_{\Gamma_h^m}\w_h\Big)_{\Gamma_h^m}, \quad i=1,\dots,s,
\end{align}
\end{subequations}
for all $\varphi_h\in S_h[\x^m]$ and $\w_h\in S_h[\x^m]^3$.

The time-node value and the next triangulated surface are then obtained from the common endpoint update \eqref{rk-pfem-update} and mesh push-forward.

\begin{lemma}[Per-stage energy decay for surface diffusion]\label{lem:perstage-sd-surf}
For each stage $i=1,\dots,s$, any solution of the nonlinear scheme \eqref{sd-BGN-High-fully_surf} satisfies
\begin{align}
\big(\nabla_{\Gamma_h^m}\u_h^{m,i},\,\nabla_{\Gamma_h^m}F_h^{m,i}\big)_{\Gamma_h^m} = -\int_{\Gamma_h^m}|\nabla_{\Gamma_h^m}H_h^{m,i}|^2 \le 0.
\end{align}
\end{lemma}

\begin{proof}
Test \eqref{sd-scheme-Xi_surf} with $\varphi_h=H_h^{m,i}$ and \eqref{sd-scheme-Hi_surf} with $\w_h=F_h^{m,i}$. The same cancellation as in Lemma~\ref{lem:perstage-sd} gives
\begin{align}
\big(\nabla_{\Gamma_h^m}\u_h^{m,i},\,\nabla_{\Gamma_h^m}F_h^{m,i}\big)_{\Gamma_h^m}
= \Big(\frac{1}{2}|\nabla_{\Gamma_h^m}\u_h^{m,i}|^2 F_h^{m,i}\cdot\n_h^{m,i},\,H_h^{m,i}\Big)^h_{\Gamma_h^m}
= -\int_{\Gamma_h^m}|\nabla_{\Gamma_h^m}H_h^{m,i}|^2.
\end{align}
\end{proof}

\begin{theorem}[Energy decay for surface diffusion]\label{thm:sd_energy_decay_surf}
Let the RK method be algebraically stable as in Theorem~\ref{thm:mcf_energy_decay_curve}. Fix $m\in\{0,\ldots,M-1\}$ and $\tau_m>0$. Assume that \eqref{sd-BGN-High-fully_surf} admits an exact solution whose intermediate triangulated surfaces are nondegenerate. Then
\begin{align}
|\Gamma_h^{m+1}|\le|\Gamma_h^m|.
\end{align}
\end{theorem}

\begin{proof}
Apply \eqref{B-stability-identity} on $S_h[\x^m]^3/\R^3$ with the gradient inner product. Lemma~\ref{lem:perstage-sd-surf} yields
\begin{align}
\frac{1}{2}\Big(\|\nabla_{\Gamma_h^m}\u_h^{m+1}\|^2_{L^2(\Gamma_h^m)} - \|\nabla_{\Gamma_h^m}\u_h^{m,0}\|^2_{L^2(\Gamma_h^m)}\Big) \le 0.
\end{align}
The surface-area inequality \eqref{energy-proof-4_surf}, valid for any piecewise-linear map, gives the stated energy decay by the argument of Theorem~\ref{thm:mcf_energy_decay_surf}.
\end{proof}

\section{Numerical results}\label{Se:5}
We present numerical results for the proposed high-order BGN methods using the Radau IIA family with $s=1,2,3$. The experiments examine temporal convergence in geometric shape distance, discrete energy decay, and mesh quality for planar curves and closed genus-$0$ surfaces. All computations are performed with Netgen/NGSolve \cite{schoeberl2014ngsolve}.

\subsection{Radau IIA methods}

Throughout the experiments, we use the $s$-stage Radau IIA methods. Their nodes $c_1<\cdots<c_s=1$ are the roots of $P_s(2x-1)-P_{s-1}(2x-1)$, where $P_s$ denotes the Legendre polynomial. Hence the rightmost node is the right endpoint of the interval. For $s=1,2,3$, the nodes and weights are
\begin{align}
& s=1:\quad c=(1),\quad b=(1),\quad A=(1),\\
& s=2:\quad c=\Big(\tfrac13,\,1\Big),\quad b=\Big(\tfrac34,\,\tfrac14\Big),\quad
A=\begin{pmatrix}\tfrac{5}{12} & -\tfrac{1}{12}\\[2pt] \tfrac{3}{4} & \tfrac{1}{4}\end{pmatrix},\\
& s=3:\quad c=\Big(\tfrac{4-\sqrt6}{10},\,\tfrac{4+\sqrt6}{10},\,1\Big),\quad
b=\Big(\tfrac{16-\sqrt6}{36},\,\tfrac{16+\sqrt6}{36},\,\tfrac19\Big),
\end{align}
and the Runge--Kutta matrix for $s=3$ is
\begin{align}
A=\begin{pmatrix}
\tfrac{88-7\sqrt6}{360} & \tfrac{296-169\sqrt6}{1800} & \tfrac{-2+3\sqrt6}{225}\\[2pt]
\tfrac{296+169\sqrt6}{1800} & \tfrac{88+7\sqrt6}{360} & \tfrac{-2-3\sqrt6}{225}\\[2pt]
\tfrac{16-\sqrt6}{36} & \tfrac{16+\sqrt6}{36} & \tfrac{1}{9}
\end{pmatrix}.
\end{align}
These methods are algebraically stable and $L$-stable, with classical order $k=2s-1$. They are stiffly accurate, $b^\top=e_s^\top A$, where $e_s=(0,\ldots,0,1)^\top\in\R^s$, so the numerical update equals the last stage value.

\subsection{Temporal convergence}

For two compact curves or surfaces $\Gamma$ and $\widetilde\Gamma$, let
\begin{align}\label{eq:hausdorff-distance}
d_{\rm H}(\Gamma,\widetilde\Gamma)
=\max\left\{
\sup_{x\in\Gamma}\inf_{\widetilde x\in\widetilde\Gamma}|x-\widetilde x|,
\sup_{\widetilde x\in\widetilde\Gamma}\inf_{x\in\Gamma}|\widetilde x-x|
\right\}.
\end{align}
On a fixed spatial mesh, let $\Gamma_{h,M}(T)$ denote the numerical approximation at time $T$ obtained with $M$ uniform time steps of size $\tau=T/M$. We define the geometric temporal self-error by
\begin{align}\label{eq:geometric-self-error}
E_M=d_{\rm H}\bigl(\Gamma_{h,M}(T),\Gamma_{h,2M}(T)\bigr).
\end{align}

\begin{example}[Curve-shortening flow of a circle and mean curvature flow of a sphere]
Under curve-shortening flow, the unit circle remains circular with exact radius $R(t)=\sqrt{1-2t}$. Starting from a regular polygon with $512$ vertices, the discrete solution preserves rotational symmetry, and its vertices therefore share a common radius $R_M(T)$. Hence, the Hausdorff self-error in \eqref{eq:geometric-self-error} reduces to $|R_M(T)-R_{2M}(T)|$. We set $T=0.2$ and use $M=2,4,8,16$. The left panel of Figure~\ref{fig:temporal-convergence} reports the resulting errors together with the fitted slopes and design-order reference lines.

Under mean curvature flow, the unit sphere remains spherical with exact radius $R(t)=\sqrt{1-4t}$. Starting from a regular icosahedron, the discrete surface remains homothetic, and the Hausdorff self-error therefore equals the difference between the common vertex radii of the two temporal approximations. At $T=0.1$, we evaluate \eqref{eq:geometric-self-error} for $M=4,8,16,32$, with each error computed from solutions using $M$ and $2M$ time steps. The right panel of Figure~\ref{fig:temporal-convergence} reports the errors together with the fitted slopes and design-order reference lines.

In both tests, the fitted slopes agree with the theoretical convergence orders.
\end{example}

\begin{figure}[htp]
\centering
\includegraphics[width=0.98\textwidth]{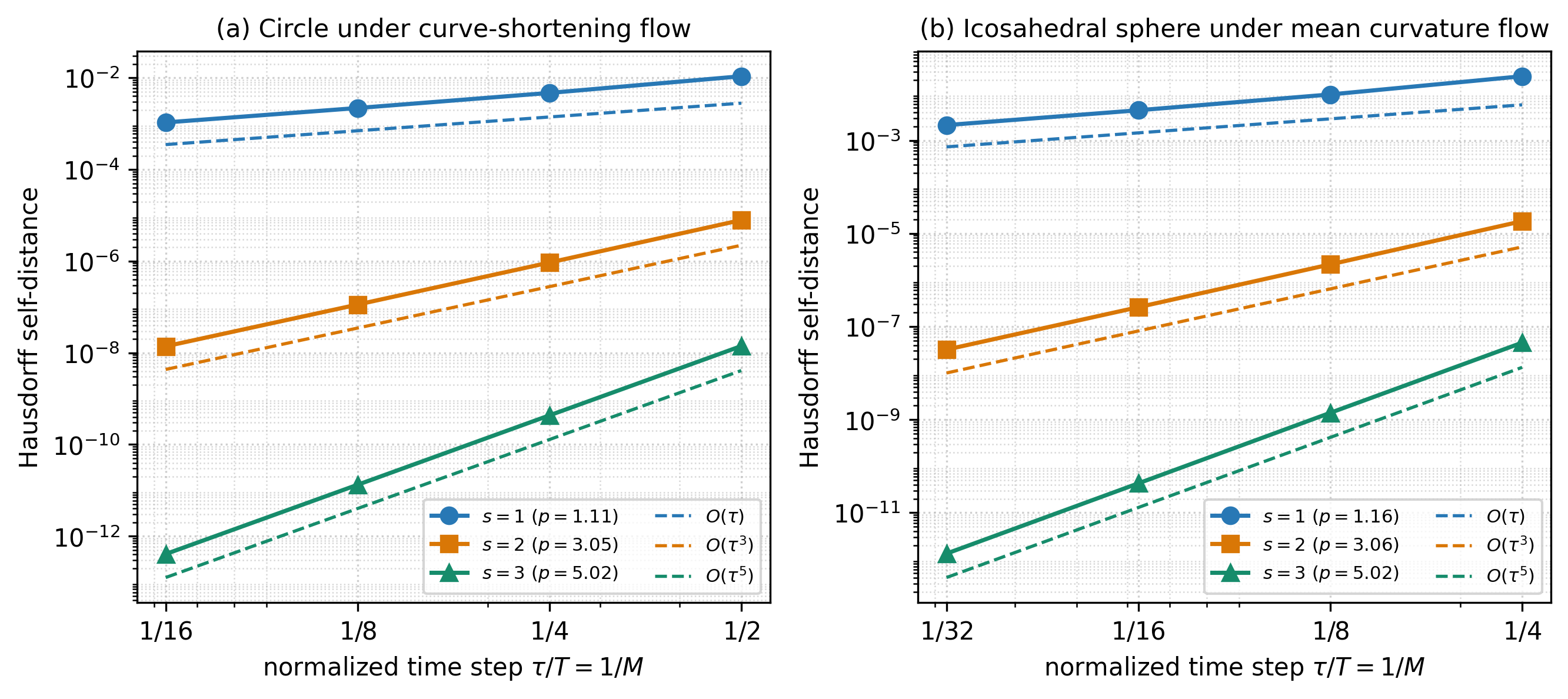}
\caption{Hausdorff self-errors for curve-shortening flow of a regular polygon (left) and mean curvature flow of a regular icosahedron (right), with dashed $O(\tau^{2s-1})$ reference lines.}
\label{fig:temporal-convergence}
\end{figure}

\begin{example}[Asymmetric curve flows]
For nonsymmetric evolving curves, the polygonal-curve distance in \eqref{eq:geometric-self-error} is evaluated directly between the two closed polygonal sets and is independent of node labels.

The two curve tests start from the smooth asymmetric curve
\begin{align}\label{eq:asymmetric-initial-curve}
r(\theta)&=1+0.16\cos(2\theta)+0.10\sin(3\theta)
             +0.05\cos(5\theta+0.3),\notag\\
\X^0(\theta)&=\big(1.10r(\theta)\cos\theta,\,
                         0.85r(\theta)\sin\theta\big),
\qquad 0\le\theta<2\pi.
\end{align}
We use $384$ vertices equidistributed in arc length and fix the phase by marking the vertex with maximal first coordinate. Curve-shortening flow is computed to $T=0.02$ with $M=1,2,4,8,16$, and curve diffusion is computed to $T=0.002$ with $M=1,2,4,8$.

The convergence rates observed in Figure~\ref{fig:geometric-temporal-convergence} agree with the theoretical convergence orders of the Radau IIA Runge--Kutta methods used in the computations.
\end{example}

\begin{figure}[htp]
\centering
\includegraphics[width=0.98\textwidth]{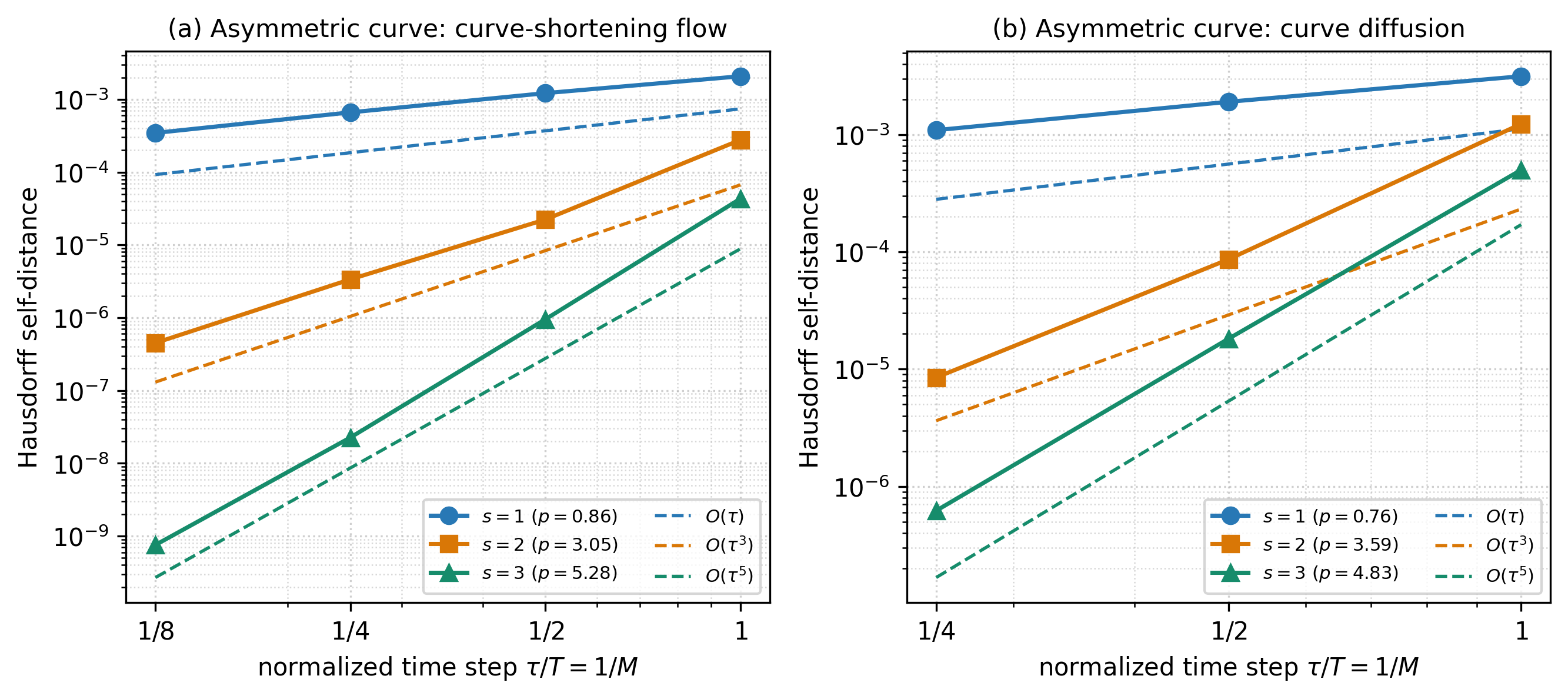}
\caption{Hausdorff self-errors for curve-shortening flow (left) and curve diffusion (right) of the asymmetric curve \eqref{eq:asymmetric-initial-curve}, plotted against the normalized time step with dashed $O(\tau^{2s-1})$ reference lines.}
\label{fig:geometric-temporal-convergence}
\end{figure}

\begin{example}[Surface diffusion of an ellipsoid]
For a nonsymmetric surface-diffusion test, we start from the ellipsoid
\begin{align}\label{eq:sd-ellipsoid}
\frac{x_1^2}{1.15^2}+x_2^2+\frac{x_3^2}{0.85^2}=1
\end{align}
and compute to $T=2.5\times10^{-4}$.
The corresponding step numbers are $(8,13,15,20,28,33,42)$ for $s=1$, $(8,10,11,13,15,16,18)$ for $s=2$, and $(8,9,10,11,12,13,14)$ for $s=3$.

The convergence rates observed in Figure~\ref{fig:ellipsoid-sd-coupled-convergence} agree with the theoretical convergence orders of the Radau IIA Runge--Kutta methods.
\end{example}

\begin{figure}[htp]
\centering
\includegraphics[width=0.78\textwidth]{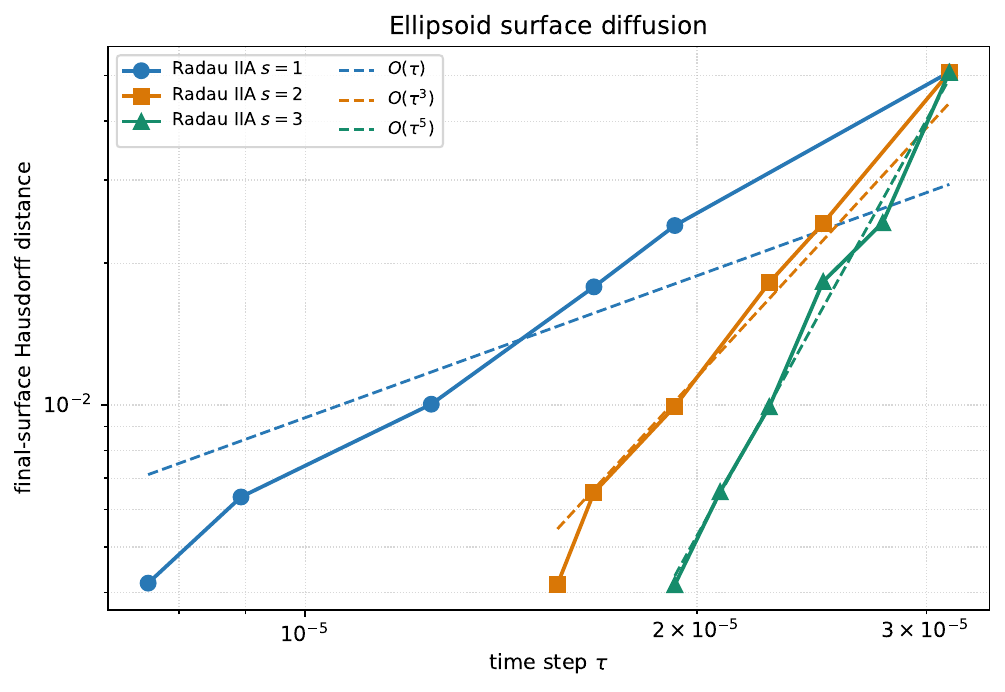}
\caption{Solid curves show the Hausdorff distances, and dashed lines indicate the reference orders $O(\tau^{2s-1})$.}
\label{fig:ellipsoid-sd-coupled-convergence}
\end{figure}

\subsection[Flower-curve flows]{Flower-curve flows}

\begin{example}[Flower curve under curve-shortening flow and curve diffusion]
\label{ex:flower-flows}
The initial flower curve is
\begin{align}\label{eq:flower-curve}
\X^0(\theta)=\bigl(1+0.65\sin(7\theta)\bigr)(\cos\theta,\sin\theta),
\qquad 0\le\theta<2\pi.
\end{align}
The initial polygon has eight edges on each half-petal, giving $112$ vertices. If $h_e^m$ is the length of edge $e$ at time $t_m$, define
\begin{align}\label{eq:mesh-quality-ratio}
Q_h^m=\frac{\max_e h_e^m}{\min_e h_e^m}.
\end{align}
We first evolve \eqref{eq:flower-curve} by curve-shortening flow to $T=0.3$ with $s=2,3$ and $800$ time steps, using the classical BGN scheme as a reference. The discrete length decreases from approximately $19.72$ to $4.83$. Figure~\ref{fig:mcf-flower} and the left panel of Figure~\ref{fig:flower-mesh-quality} show the evolving curve and $Q_h^m$. Its initial value is approximately $13$, and its final value is approximately $1.15$ for the Radau computations and $1.22$ for the BGN reference.

For curve diffusion, we use $T=0.04$, $s=2,3$, $112$ vertices, and $800$ time steps. The discrete length decreases from approximately $19.72$ to $6.85$, while $Q_h^m$ decreases from approximately $13$ to $1.08$ for both Radau computations and to $1.27$ for the BGN reference. Figure~\ref{fig:sd-flower} and the right panel of Figure~\ref{fig:flower-mesh-quality} show the corresponding curve evolution and mesh-quality ratio. After the initial redistribution, the Radau IIA curves lie below the BGN reference in both mesh-quality plots. Since $Q_h^m=1$ corresponds to equal edge lengths, the proposed high-order BGN schemes produce more uniform meshes than the classical BGN scheme in these flower-curve experiments.
\end{example}

\begin{figure}[htp]
\centering
\includegraphics[width=0.98\textwidth]{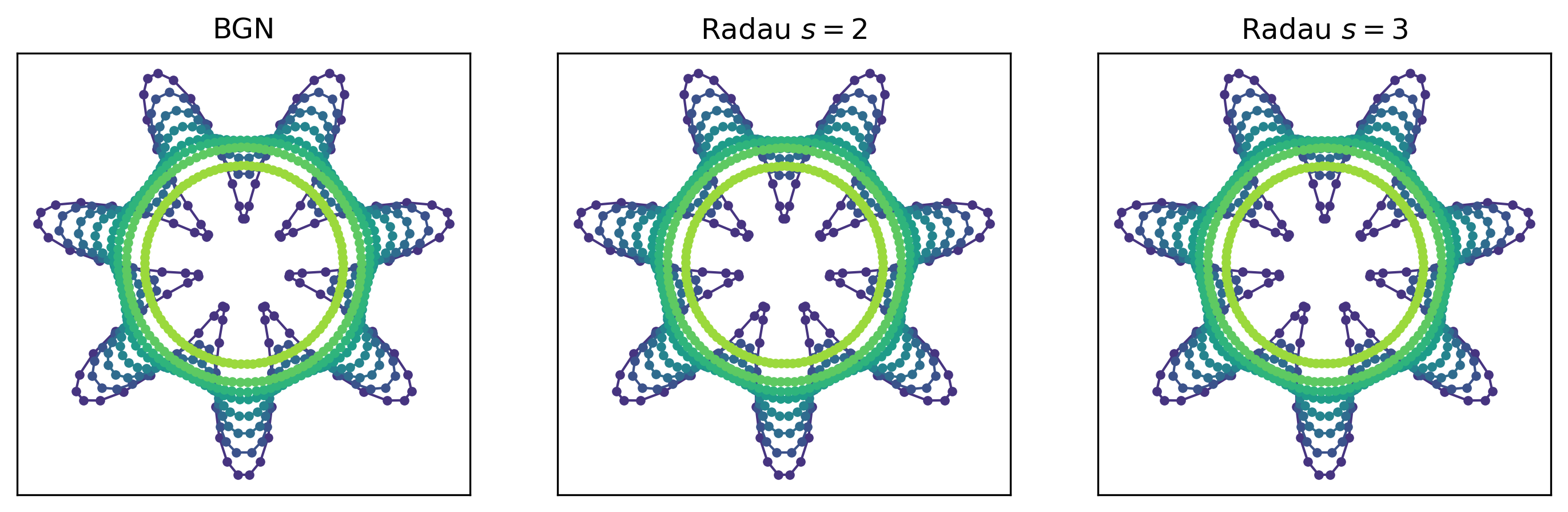}
\caption{Curve-shortening flow of \eqref{eq:flower-curve} with $112$ vertices and $\tau=0.3/800$.}
\label{fig:mcf-flower}
\end{figure}

\begin{figure}[htp]
\centering
\includegraphics[width=0.98\textwidth]{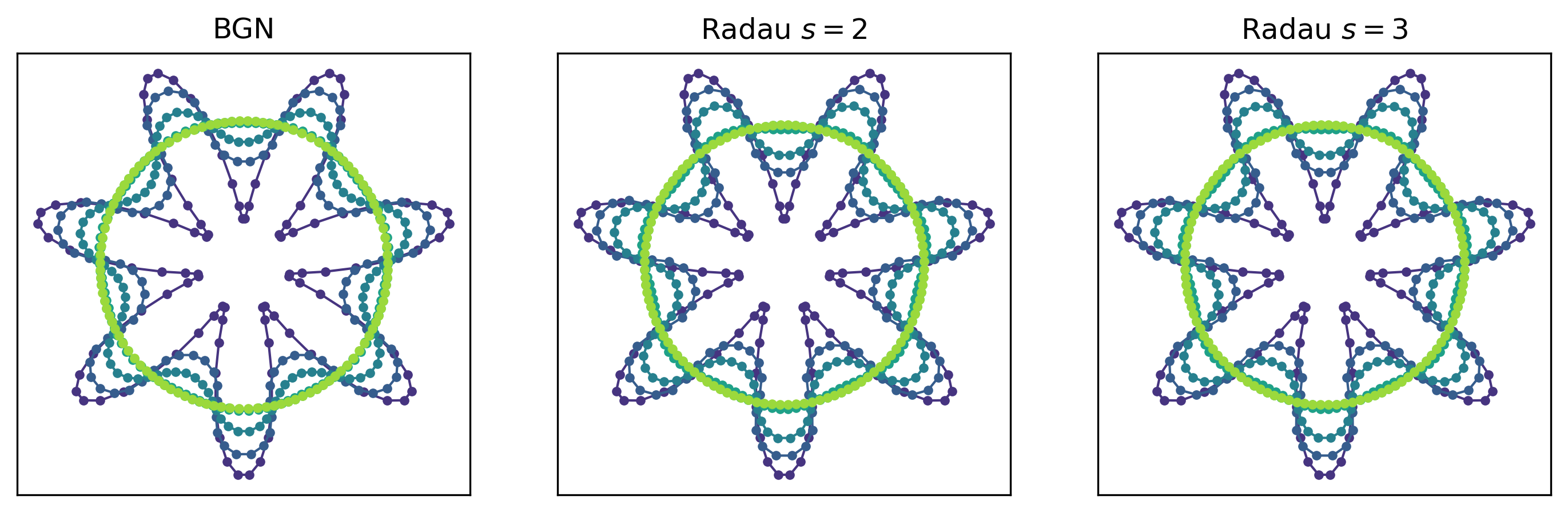}
\caption{Curve diffusion of \eqref{eq:flower-curve}. The classical BGN reference and the nonlinear Radau IIA schemes with $s=2,3$ use $112$ vertices and $\tau=0.04/800$.}
\label{fig:sd-flower}
\end{figure}

\begin{figure}[htp]
\centering
\includegraphics[width=0.49\textwidth]{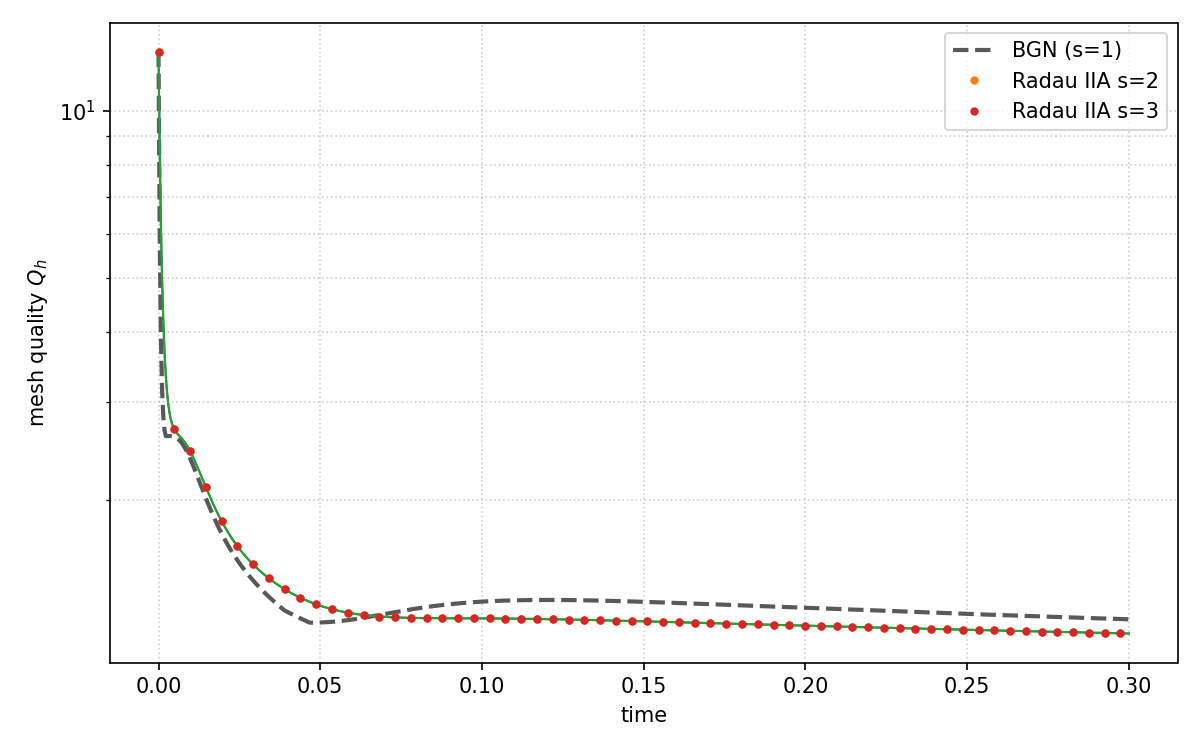}\hfill
\includegraphics[width=0.49\textwidth]{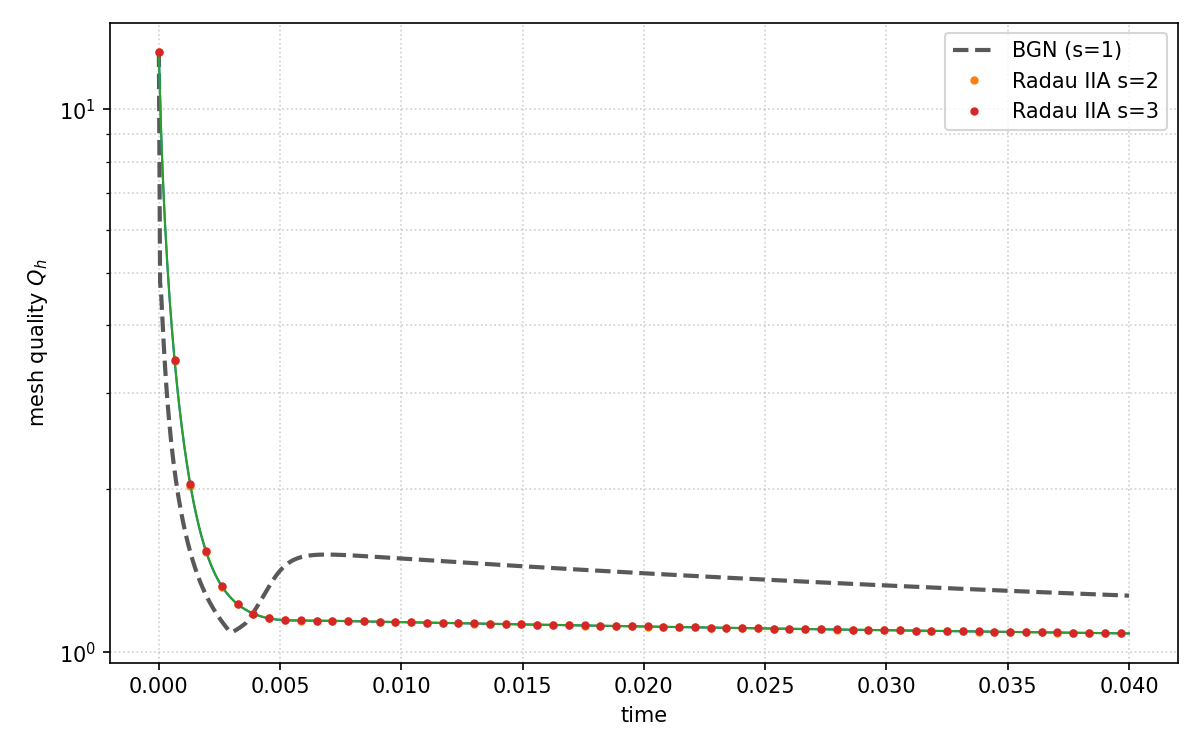}
\caption{Mesh-quality ratio $Q_h^m$ for curve-shortening flow (left) and curve diffusion (right). Dashed curves show the classical BGN results.}
\label{fig:flower-mesh-quality}
\end{figure}

\subsection[Surface flows of a cube]{Surface flows of a cube}

We apply the surface schemes of Section~\ref{Se:4} to the cube $[-1,1]^3$. Triangulation edges are displayed in every snapshot. All snapshots in each evolution use the same coordinate box and camera.

\begin{example}[Cube under mean curvature flow and surface diffusion]
\label{ex:cube-flows}
The cube has initial area $24$. Each face has six subdivisions along each edge and is triangulated by splitting every grid cell into two triangles. For mean curvature flow, we use $200$ time steps up to $T=0.2$. Figure~\ref{fig:cube-mcf-mesh} follows the rounding and shrinking of the $s=3$ mesh, and the left panel of Figure~\ref{fig:cube-area-curves} shows the area histories for $s=1,2,3$. The final areas are $8.7390$, $8.7334$, and $8.7333$, respectively. The final surface is nearly spherical.

For surface diffusion, we use $100$ time steps up to $T=0.06$. Figure~\ref{fig:cube-sd-mesh} and the right panel of Figure~\ref{fig:cube-area-curves} show the computations with $s=1,2,3$. The final areas are $19.4515$, $19.3830$, and $19.3787$, respectively.
\end{example}

\begin{figure}[htp]
\centering
\includegraphics[width=0.24\textwidth]{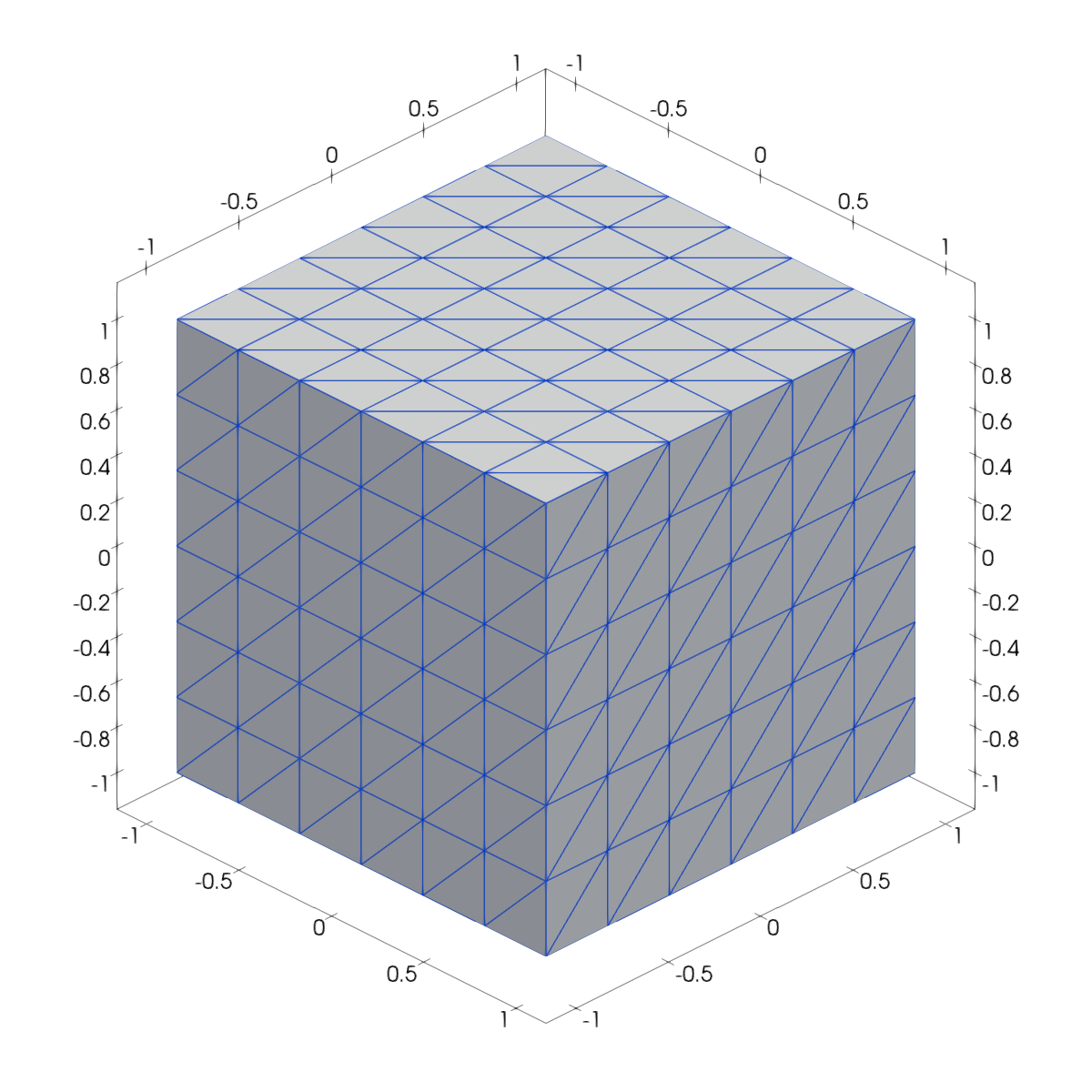}\hfill
\includegraphics[width=0.24\textwidth]{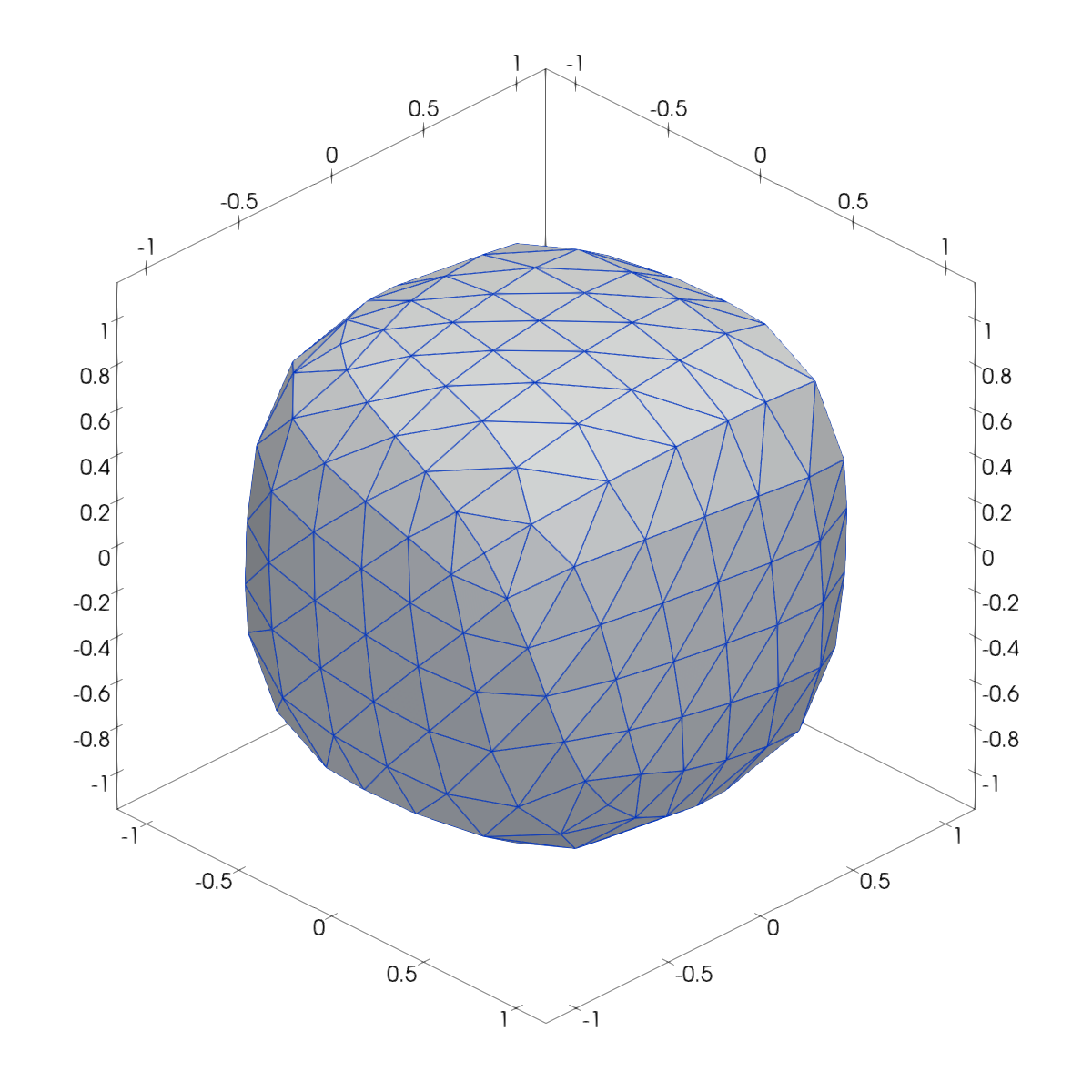}\hfill
\includegraphics[width=0.24\textwidth]{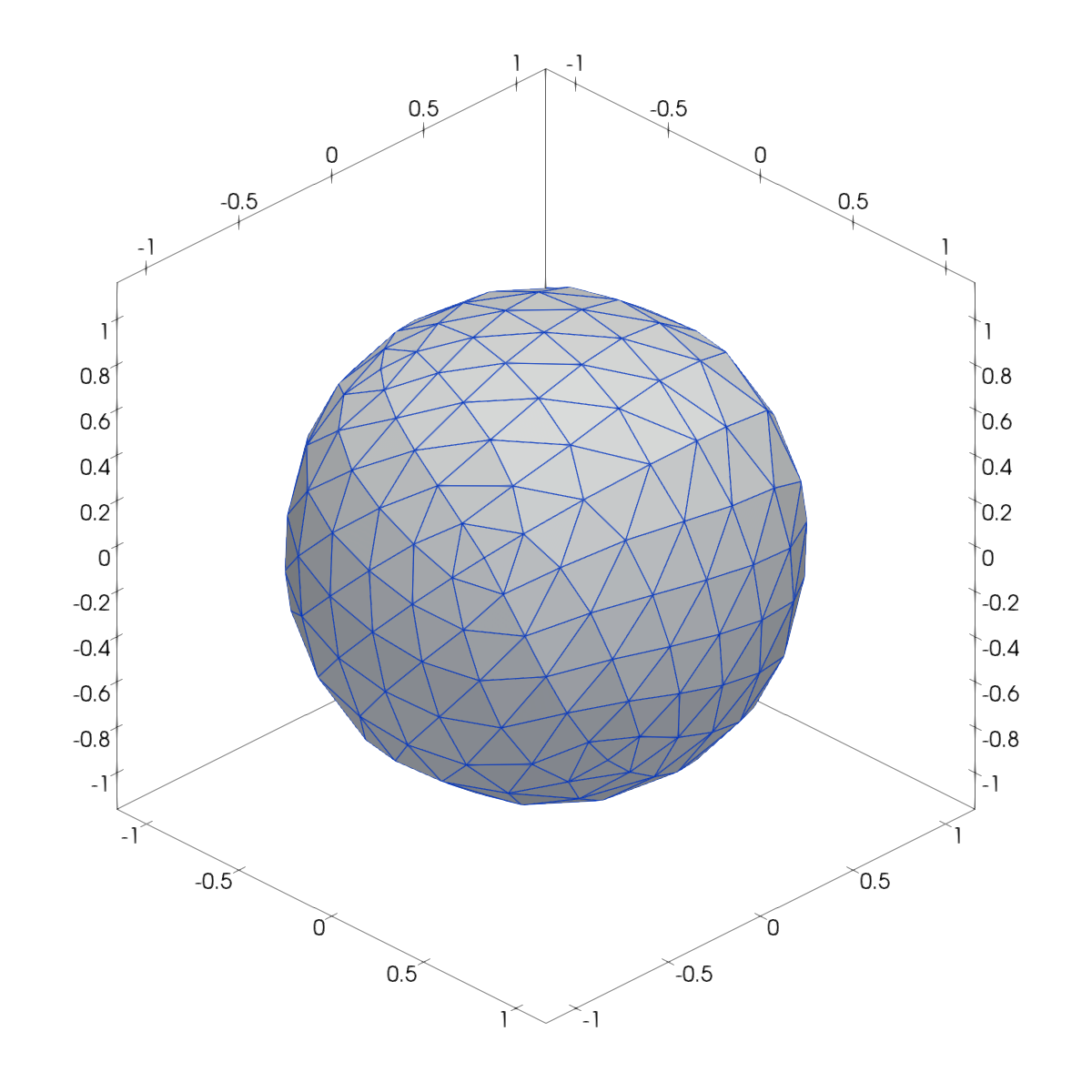}\hfill
\includegraphics[width=0.24\textwidth]{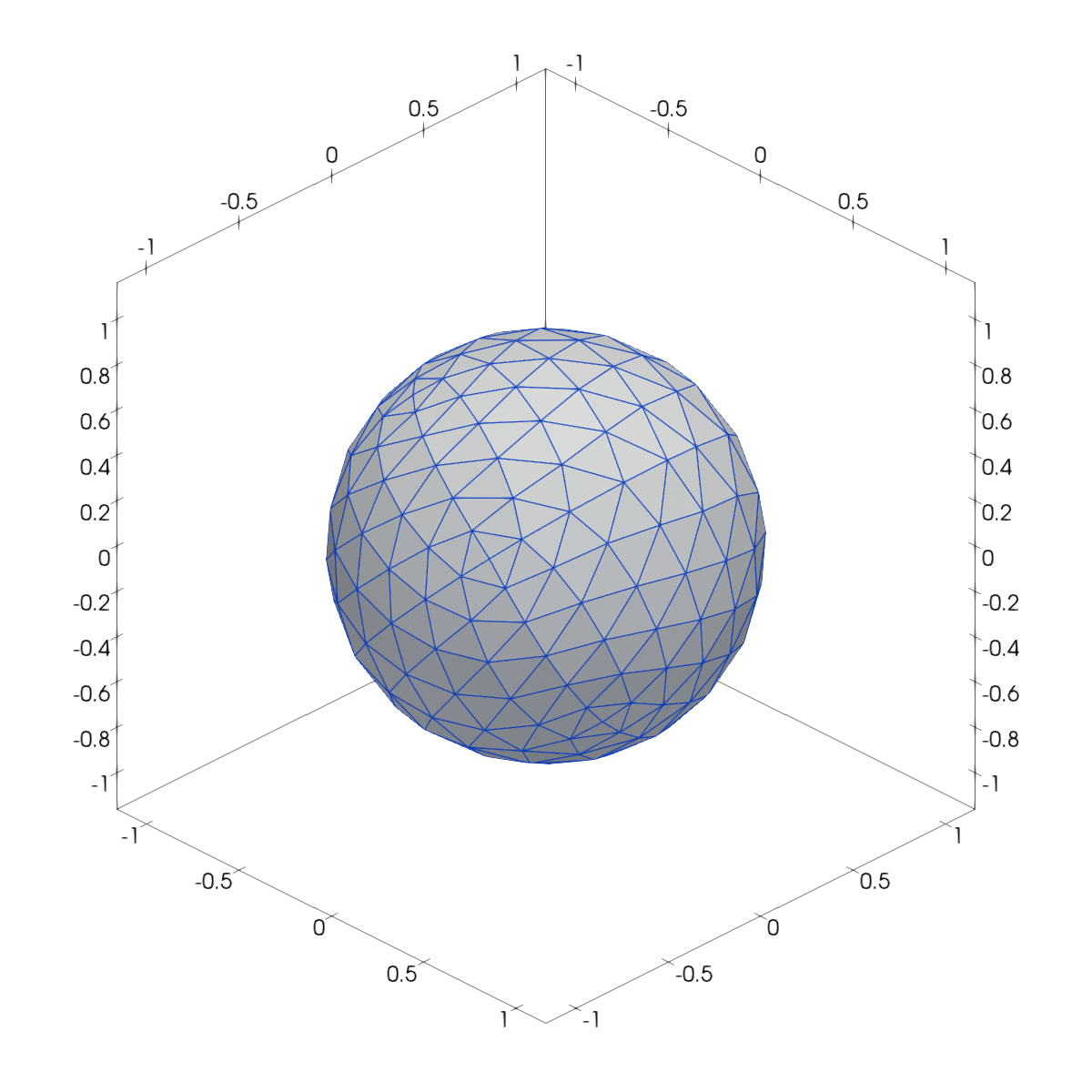}
\caption{Mean curvature flow of the cube. Mesh snapshots from the $s=3$ computation are shown at $t=0$, $0.067$, $0.133$, and $0.2$ from left to right.}
\label{fig:cube-mcf-mesh}
\end{figure}

\begin{figure}[htp]
\centering
\includegraphics[width=0.24\textwidth]{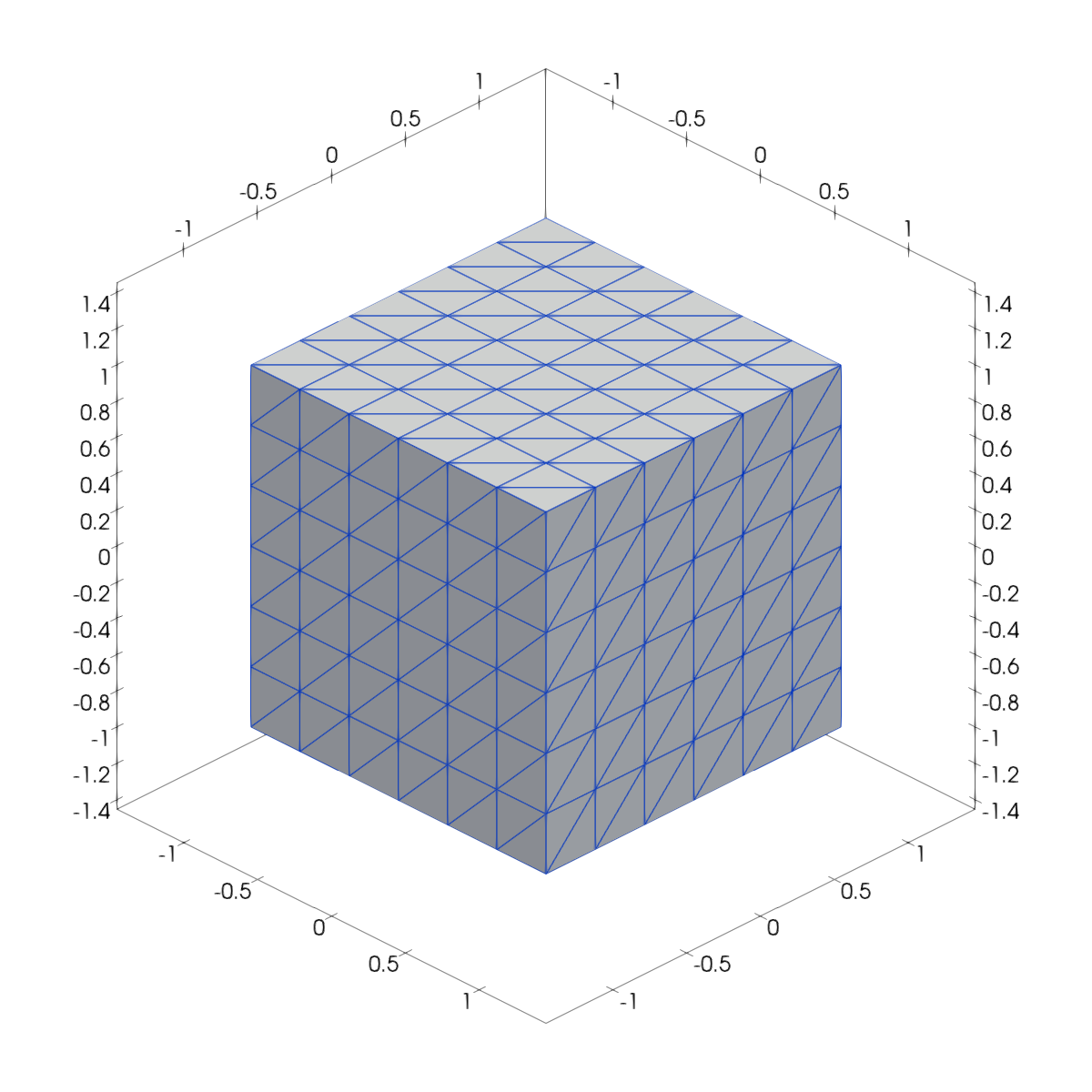}\hfill
\includegraphics[width=0.24\textwidth]{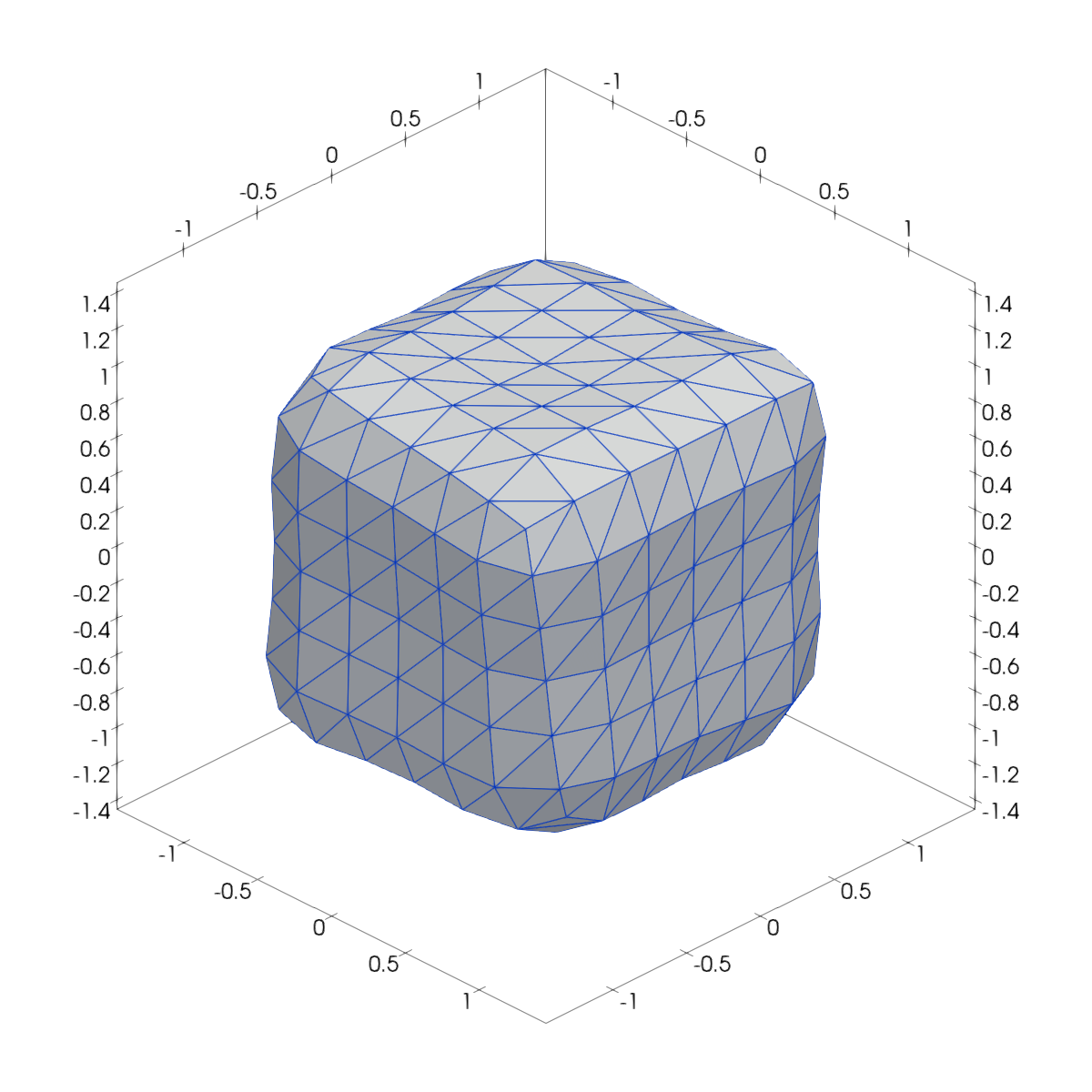}\hfill
\includegraphics[width=0.24\textwidth]{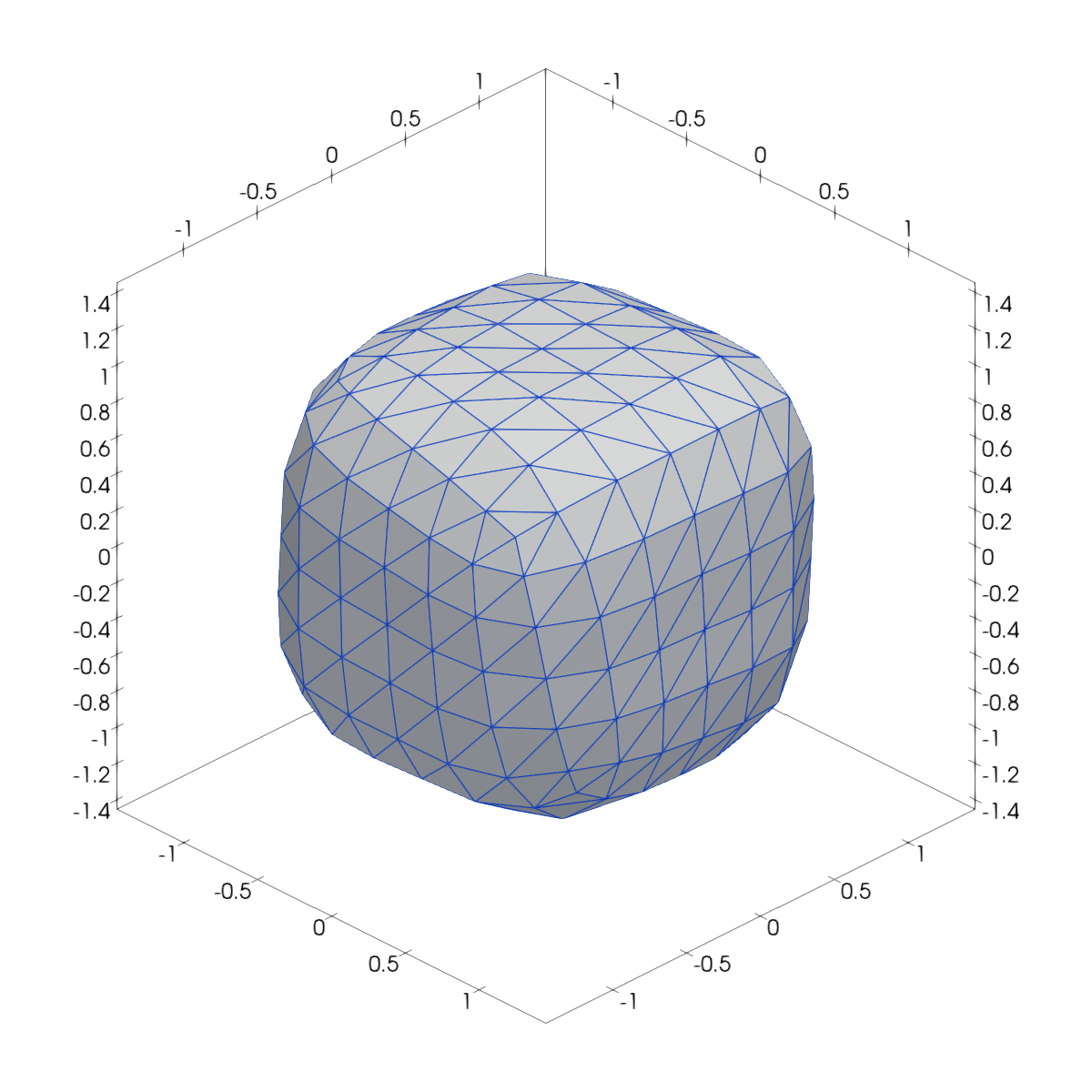}\hfill
\includegraphics[width=0.24\textwidth]{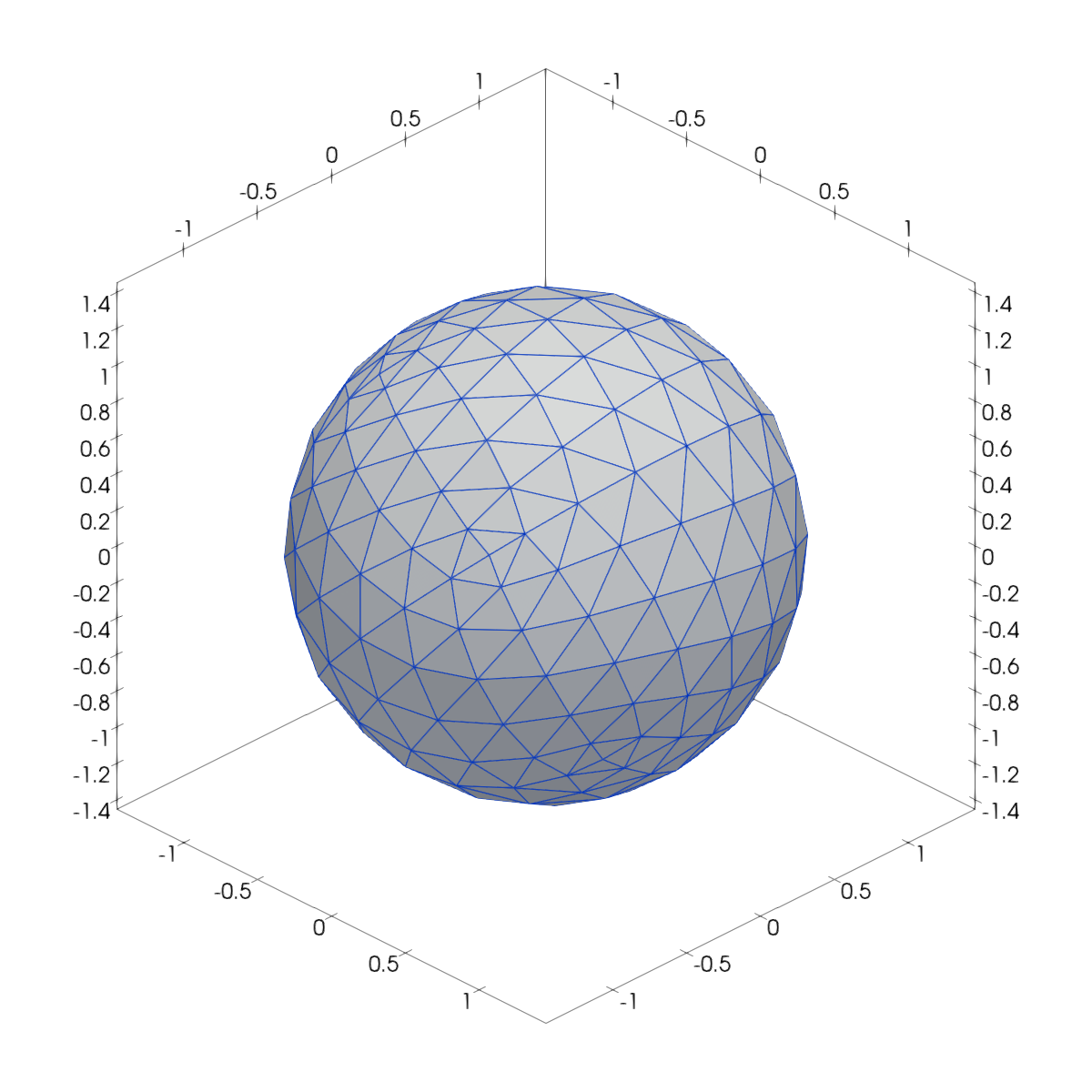}
\caption{Surface diffusion of the cube. Mesh snapshots from the $s=3$ computation are shown at $t=0$, $0.0018$, $0.006$, and $0.06$ from left to right.}
\label{fig:cube-sd-mesh}
\end{figure}

\begin{figure}[htp]
\centering
\includegraphics[width=0.49\textwidth]{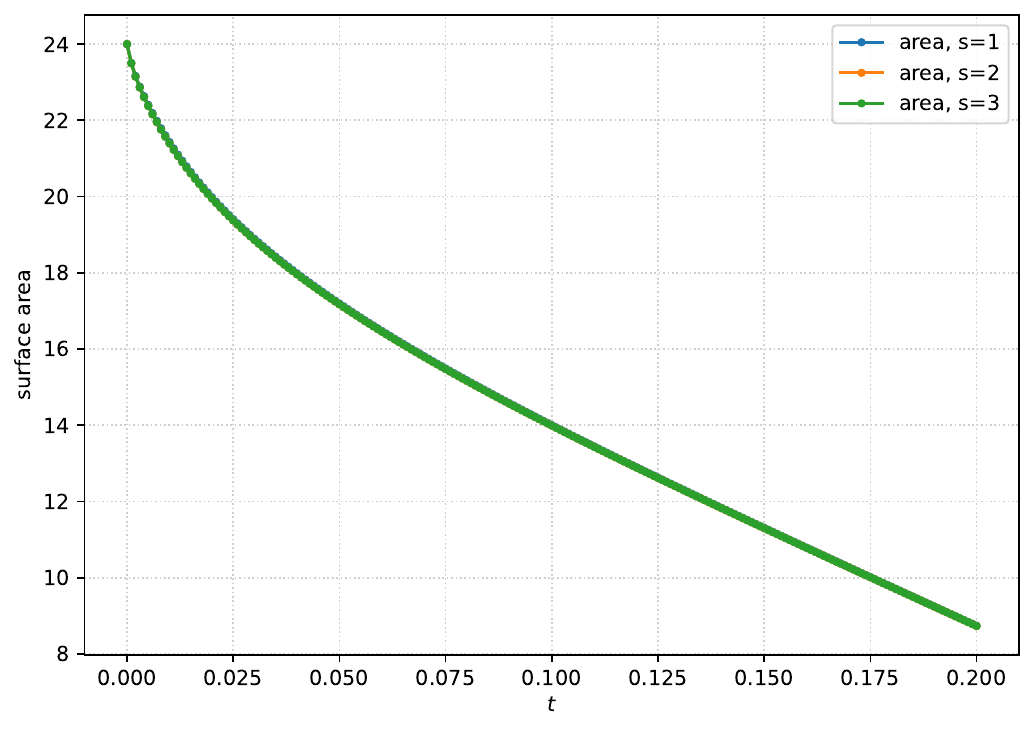}\hfill
\includegraphics[width=0.49\textwidth]{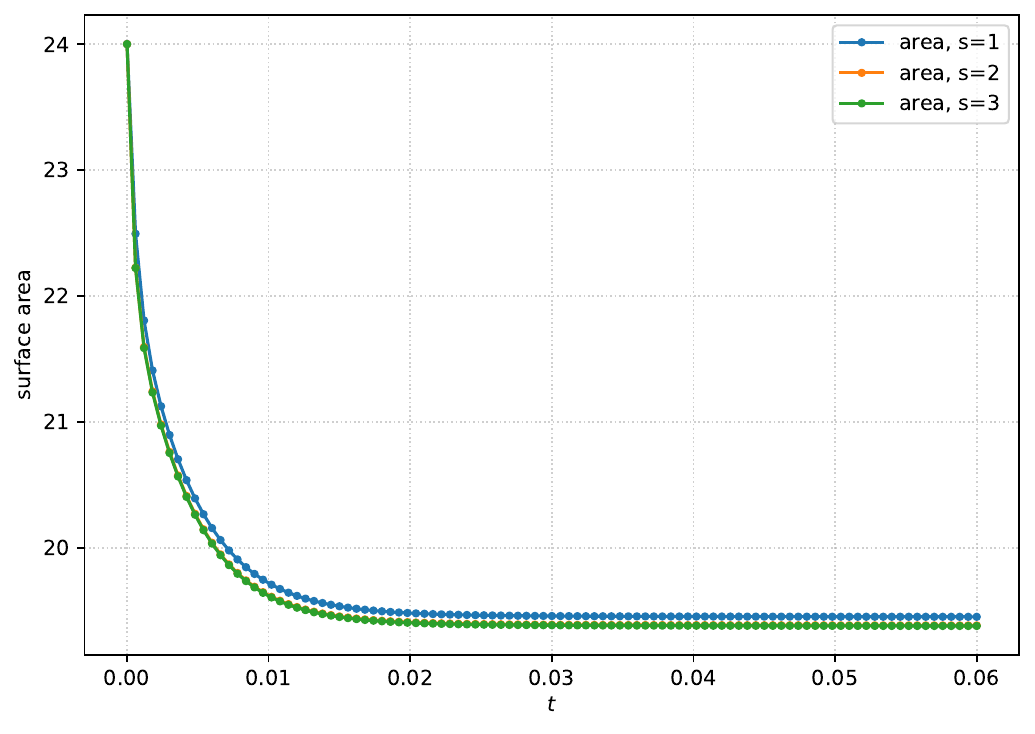}
\caption{Surface area versus time for mean curvature flow (left) and surface diffusion (right) of the cube. All curves are monotone.}
\label{fig:cube-area-curves}
\end{figure}

\section{Conclusion}\label{Se:6}

On each time slab, the harmonic-map equations are rewritten on the left-endpoint geometry $\Gamma^m$, while the Runge--Kutta internal times determine the intermediate target geometries. For closed genus-$0$ surfaces, conformality identifies the metric density in this common-domain formulation. Internal-time evaluation followed by mass-lumped parametric finite elements yields high-order Runge--Kutta extensions of the paired BGN structure for the four flows. Algebraic stability with nonnegative weights, stagewise BGN energy decay, and the algebraic Runge--Kutta identity give one-step decay of discrete length or area whenever an exact nondegenerate stage solution exists. The Radau IIA experiments exhibit this decay, BGN-type mesh redistribution for curves, and high-order self-convergence behavior consistent with the corresponding design orders.

\backmatter

\bibliography{surface-reference-SN,v5-related-references-SN}

@article{bao2021structure,
  title={A structure-preserving parametric finite element method for surface diffusion},
  author={Bao, Weizhu and Zhao, Quan},
  journal={SIAM Journal on Numerical Analysis},
  volume={59},
  number={5},
  pages={2775--2799},
  year={2021},
  publisher={SIAM}
}

@article{zhao2021energy,
  title={An energy-stable parametric finite element method for simulating solid-state dewetting},
  author={Zhao, Quan and Jiang, Wei and Bao, Weizhu},
  journal={IMA Journal of Numerical Analysis},
  volume={41},
  number={3},
  pages={2026--2055},
  year={2021},
  publisher={Oxford University Press}
}

@article{dziuk1990algorithm,
  title={An algorithm for evolutionary surfaces},
  author={Dziuk, Gerhard},
  journal={Numerische Mathematik},
  volume={58},
  pages={603--611},
  year={1990},
  publisher={Springer},
  doi={10.1007/BF01385643}
}

@article{barrett2007variational,
  title={On the variational approximation of combined second and fourth order geometric evolution equations},
  author={Barrett, John W and Garcke, Harald and N{\"u}rnberg, Robert},
  journal={SIAM Journal on Scientific Computing},
  volume={29},
  number={3},
  pages={1006--1041},
  year={2007},
  publisher={SIAM}
}

@article{barrett2007parametric,
  title={A parametric finite element method for fourth order geometric evolution equations},
  author={Barrett, John W and Garcke, Harald and N{\"u}rnberg, Robert},
  journal={Journal of Computational Physics},
  volume={222},
  number={1},
  pages={441--467},
  year={2007},
  publisher={Elsevier}
}

@article{duan2024new,
  title={New artificial tangential motions for parametric finite element approximation of surface evolution},
  author={Duan, Beiping and Li, Buyang},
  journal={SIAM Journal on Scientific Computing},
  volume={46},
  number={1},
  pages={A587--A608},
  year={2024},
  publisher={SIAM}
}

@article{jiang2024second,
  title={A second-order in time, {BGN}-based parametric finite element method for geometric flows of curves},
  author={Jiang, Wei and Su, Chunmei and Zhang, Ganghui},
  journal={Journal of Computational Physics},
  volume={514},
  pages={113220},
  year={2024},
  publisher={Elsevier},
  doi={10.1016/j.jcp.2024.113220}
}

@article{hu2022evolving,
  title={Evolving finite element methods with an artificial tangential velocity for mean curvature flow and {Willmore} flow},
  author={Hu, Jiashun and Li, Buyang},
  journal={Numerische Mathematik},
  volume={152},
  number={1},
  pages={127--181},
  year={2022},
  publisher={Springer}
}

@article{barrett2008parametric,
  title={On the parametric finite element approximation of evolving hypersurfaces in {$\mathbb{R}^3$}},
  author={Barrett, John W and Garcke, Harald and N{\"u}rnberg, Robert},
  journal={Journal of Computational Physics},
  volume={227},
  number={9},
  pages={4281--4307},
  year={2008},
  publisher={Elsevier}
}

@article{dziukLubichMansour2012,
  title={Runge--Kutta Time Discretization of Parabolic Differential Equations on Evolving Surfaces},
  author={Dziuk, Gerhard and Lubich, Christian and Mansour, Dhia},
  journal={IMA Journal of Numerical Analysis},
  volume={32},
  number={2},
  pages={394--416},
  year={2012},
  doi={10.1093/imanum/drr017},
  publisher={Oxford University Press}
}

@article{kovacs2018higher,
  title={Higher order time discretizations with ALE finite elements for parabolic problems on evolving surfaces},
  author={Kov{\'a}cs, Bal{\'a}zs and Power Guerra, Christian Andreas},
  journal={IMA Journal of Numerical Analysis},
  volume={38},
  number={1},
  pages={460--494},
  year={2018},
  doi={10.1093/imanum/drw074},
  publisher={Oxford University Press}
}

@techreport{schoeberl2014ngsolve,
  author={Sch{\"o}berl, Joachim},
  title={{C++11} Implementation of Finite Elements in {NGSolve}},
  institution={Institute for Analysis and Scientific Computing, Vienna University of Technology},
  type={ASC Report},
  number={30/2014},
  year={2014}
}

@book{hairer2006solving,
  title={Solving Ordinary Differential Equations II: Stiff and Differential-Algebraic Problems},
  author={Hairer, Ernst and Wanner, Gerhard},
  series={Springer Series in Computational Mathematics},
  volume={14},
  edition={2nd},
  year={1996},
  publisher={Springer},
  address={Berlin}
}

@article{dziuk1994convergence,
  author  = {Dziuk, Gerhard},
  title   = {Convergence of a Semi-discrete Scheme for the Curve Shortening Flow},
  journal = {Mathematical Models and Methods in Applied Sciences},
  volume  = {4},
  number  = {4},
  pages   = {589--606},
  year    = {1994},
  doi     = {10.1142/S0218202594000339}
}

@article{banschMorinNochetto2005parametric,
  author  = {B{\"a}nsch, Eberhard and Morin, Pedro and Nochetto, Ricardo H.},
  title   = {A Finite Element Method for Surface Diffusion: The Parametric Case},
  journal = {Journal of Computational Physics},
  volume  = {203},
  number  = {1},
  pages   = {321--343},
  year    = {2005},
  doi     = {10.1016/j.jcp.2004.08.022}
}

@incollection{bgn2020review,
  author    = {Barrett, John W. and Garcke, Harald and N{\"u}rnberg, Robert},
  title     = {Parametric Finite Element Approximations of Curvature-driven Interface Evolutions},
  booktitle = {Handbook of Numerical Analysis},
  volume    = {21},
  pages     = {275--423},
  publisher = {Elsevier},
  address   = {Amsterdam},
  year      = {2020},
  doi       = {10.1016/bs.hna.2019.05.002}
}

@article{li2020dziuk,
  author  = {Li, Buyang},
  title   = {Convergence of {Dziuk}'s Linearly Implicit Parametric Finite Element Method for Curve Shortening Flow},
  journal = {SIAM Journal on Numerical Analysis},
  volume  = {58},
  number  = {4},
  pages   = {2315--2333},
  year    = {2020},
  doi     = {10.1137/19M1305483}
}

@article{yeCui2021dziuk,
  author  = {Ye, Changqing and Cui, Junzhi},
  title   = {Convergence of {Dziuk}'s Fully Discrete Linearly Implicit Scheme for Curve Shortening Flow},
  journal = {SIAM Journal on Numerical Analysis},
  volume  = {59},
  number  = {6},
  pages   = {2823--2842},
  year    = {2021},
  doi     = {10.1137/21M1391626}
}

@article{kovacsLiLubich2019mcf,
  author  = {Kov{\'a}cs, Bal{\'a}zs and Li, Buyang and Lubich, Christian},
  title   = {A Convergent Evolving Finite Element Algorithm for Mean Curvature Flow of Closed Surfaces},
  journal = {Numerische Mathematik},
  volume  = {143},
  pages   = {797--853},
  year    = {2019},
  doi     = {10.1007/s00211-019-01074-2}
}

@article{baiLi2024analysis,
  author  = {Bai, Genming and Li, Buyang},
  title   = {A New Approach to the Analysis of Parametric Finite Element Approximations to Mean Curvature Flow},
  journal = {Foundations of Computational Mathematics},
  volume  = {24},
  pages   = {1673--1737},
  year    = {2024},
  doi     = {10.1007/s10208-023-09622-x}
}

@article{baiLi2025bgn,
  author  = {Bai, Genming and Li, Buyang},
  title   = {Convergence of a Stabilized Parametric Finite Element Method of the {Barrett--Garcke--N{\"u}rnberg} Type for Curve Shortening Flow},
  journal = {Mathematics of Computation},
  volume  = {94},
  number  = {355},
  pages   = {2151--2220},
  year    = {2025},
  doi     = {10.1090/mcom/4019}
}

@article{baiHuLi2024artificial,
  author  = {Bai, Genming and Hu, Jiashun and Li, Buyang},
  title   = {A Convergent Evolving Finite Element Method with Artificial Tangential Motion for Surface Evolution under a Prescribed Velocity Field},
  journal = {SIAM Journal on Numerical Analysis},
  volume  = {62},
  pages   = {2172--2195},
  year    = {2024},
  doi     = {10.1137/23M156968X}
}

@article{elliottFritz2016mesh,
  author  = {Elliott, Charles M. and Fritz, Hans},
  title   = {On Algorithms with Good Mesh Properties for Problems with Moving Boundaries Based on the Harmonic Map Heat Flow and the {DeTurck} Trick},
  journal = {SMAI Journal of Computational Mathematics},
  volume  = {2},
  pages   = {141--176},
  year    = {2016},
  doi     = {10.5802/smai-jcm.12}
}

@article{elliottFritz2017deturck,
  author  = {Elliott, Charles M. and Fritz, Hans},
  title   = {On Approximations of the Curve Shortening Flow and of the Mean Curvature Flow Based on the {DeTurck} Trick},
  journal = {IMA Journal of Numerical Analysis},
  volume  = {37},
  number  = {2},
  pages   = {543--603},
  year    = {2017},
  doi     = {10.1093/imanum/drw020}
}

@incollection{heleinWood2008harmonic,
  author    = {H{\'e}lein, Fr{\'e}d{\'e}ric and Wood, John C.},
  title     = {Harmonic Maps},
  booktitle = {Handbook of Global Analysis},
  editor    = {Krupka, Demeter and Saunders, David},
  publisher = {Elsevier},
  address   = {Amsterdam},
  pages     = {417--491},
  year      = {2008},
  doi       = {10.1016/B978-044452833-9.50009-7}
}

@article{duan2024energy,
  author  = {Duan, Beiping},
  title   = {Energy-Stable and Mesh-Preserving Parametric {FEM} for Mean Curvature Flow of Surfaces},
  journal = {SIAM Journal on Scientific Computing},
  volume  = {46},
  number  = {6},
  pages   = {A3873--A3896},
  year    = {2024},
  doi     = {10.1137/24M1647813}
}

@article{duan2025mesh,
  author  = {Duan, Beiping},
  title   = {Mesh-Preserving and Energy-Stable Parametric {FEM} for Geometric Flows of Surfaces},
  journal = {SIAM Journal on Numerical Analysis},
  volume  = {63},
  number  = {2},
  pages   = {619--640},
  year    = {2025},
  doi     = {10.1137/24M1671542}
}

@article{gaoLi2025geometric,
  author  = {Gao, Guangwei and Li, Buyang},
  title   = {Geometric-structure Preserving Methods for Surface Evolution in Curvature Flows with Minimal Deformation Formulations},
  journal = {Journal of Computational Physics},
  volume  = {524},
  pages   = {113718},
  year    = {2025},
  doi     = {10.1016/j.jcp.2025.113718}
}

@article{gaoGarckeLiTang2026mdr,
  author  = {Gao, Guangwei and Garcke, Harald and Li, Buyang and Tang, Rong},
  title   = {An Energy-Stable Minimal Deformation Rate Scheme for Mean Curvature Flow and Surface Diffusion},
  journal = {SIAM Journal on Scientific Computing},
  volume  = {48},
  number  = {1},
  pages   = {A103--A131},
  year    = {2026},
  doi     = {10.1137/25M1753838}
}

@misc{duanYang2026second,
  author        = {Duan, Beiping and Yang, Zongze},
  title         = {A Second-order Structure-preserving Parametric {FEM} for Surface Evolution},
  year          = {2026},
  eprint        = {2606.08293},
  archiveprefix = {arXiv},
  primaryclass  = {math.NA},
  url           = {https://arxiv.org/abs/2606.08293},
  note          = {Preprint, arXiv:2606.08293}
}

@article{jiangSuZhang2024bdf,
  author  = {Jiang, Wei and Su, Chunmei and Zhang, Ganghui},
  title   = {Stable Backward Differentiation Formula Time Discretization of {BGN}-Based Parametric Finite Element Methods for Geometric Flows},
  journal = {SIAM Journal on Scientific Computing},
  volume  = {46},
  number  = {5},
  pages   = {A2874--A2898},
  year    = {2024},
  doi     = {10.1137/23M1625597}
}

@article{jiangSuZhangZhang2025predictor,
  author  = {Jiang, Wei and Su, Chunmei and Zhang, Ganghui and Zhang, Lian},
  title   = {Predictor-corrector, {BGN}-based Parametric Finite Element Methods for Surface Diffusion},
  journal = {Journal of Computational Physics},
  volume  = {530},
  pages   = {113901},
  year    = {2025},
  doi     = {10.1016/j.jcp.2025.113901}
}

@article{deckelnickNurnberg2026second,
  author  = {Deckelnick, Klaus and N{\"u}rnberg, Robert},
  title   = {Second Order in Time Finite Element Schemes for Curve Shortening Flow and Curve Diffusion},
  journal = {SIAM Journal on Numerical Analysis},
  volume  = {64},
  number  = {1},
  pages   = {103--124},
  year    = {2026},
  doi     = {10.1137/25M1737523}
}

@article{jiangLi2021area,
  author  = {Jiang, Wei and Li, Buyang},
  title   = {A Perimeter-decreasing and Area-conserving Algorithm for Surface Diffusion Flow of Curves},
  journal = {Journal of Computational Physics},
  volume  = {443},
  pages   = {110531},
  year    = {2021},
  doi     = {10.1016/j.jcp.2021.110531}
}

@article{garckeJiangSuZhang2025lagrange,
  author  = {Garcke, Harald and Jiang, Wei and Su, Chunmei and Zhang, Ganghui},
  title   = {Structure-Preserving Parametric Finite Element Method for Surface Diffusion Based on {Lagrange} Multiplier Approaches},
  journal = {SIAM Journal on Scientific Computing},
  volume  = {47},
  number  = {3},
  pages   = {A1983--A2011},
  year    = {2025},
  doi     = {10.1137/24M1687546}
}

@article{garckeNurnbergPraetoriusZhang2025isoparametric,
  author  = {Garcke, Harald and N{\"u}rnberg, Robert and Praetorius, Simon and Zhang, Ganghui},
  title   = {Isoparametric Finite Element Methods for Mean Curvature Flow and Surface Diffusion},
  journal = {Journal of Computational Physics},
  volume  = {539},
  pages   = {114248},
  year    = {2025},
  doi     = {10.1016/j.jcp.2025.114248}
}

@article{duanLiZhang2021highorder,
  author  = {Duan, Beiping and Li, Buyang and Zhang, Zhimin},
  title   = {High-Order Fully Discrete Energy Diminishing Evolving Surface Finite Element Methods for a Class of Geometric Curvature Flows},
  journal = {Annals of Applied Mathematics},
  volume  = {37},
  number  = {4},
  pages   = {405--436},
  year    = {2021},
  doi     = {10.4208/aam.OA-2021-0007}
}

@misc{zhangAndrewsFarrell2026arbitrary,
  author        = {Zhang, Ganghui and Andrews, Boris D. and Farrell, Patrick E.},
  title         = {Arbitrary-order Structure-preserving Discretizations for Geometric Curvature Flows},
  year          = {2026},
  eprint        = {2605.20371},
  archiveprefix = {arXiv},
  primaryclass  = {math.NA},
  url           = {https://arxiv.org/abs/2605.20371},
  note          = {Preprint, arXiv:2605.20371}
}

@article{butcher1975stability,
  author  = {Butcher, John C.},
  title   = {A Stability Property of Implicit {Runge--Kutta} Methods},
  journal = {BIT},
  volume  = {15},
  pages   = {358--361},
  year    = {1975},
  doi     = {10.1007/BF01931672}
}

@article{burrageButcher1979stability,
  author  = {Burrage, Kevin and Butcher, John C.},
  title   = {Stability Criteria for Implicit {Runge--Kutta} Methods},
  journal = {SIAM Journal on Numerical Analysis},
  volume  = {16},
  number  = {1},
  pages   = {46--57},
  year    = {1979},
  doi     = {10.1137/0716004}
}

@article{hairerLubich2014energy,
  author  = {Hairer, Ernst and Lubich, Christian},
  title   = {Energy-diminishing Integration of Gradient Systems},
  journal = {IMA Journal of Numerical Analysis},
  volume  = {34},
  number  = {2},
  pages   = {452--461},
  year    = {2014},
  doi     = {10.1093/imanum/drt031}
}

\end{document}